\documentclass[a4paper,reqno, a4wide]{amsart}
\usepackage[utf8]{inputenc}             
\usepackage[T1]{fontenc}                
\usepackage[numbers]{natbib}
\usepackage[english]{babel}              
\usepackage{graphicx}                   
\usepackage{amsmath,amsfonts,amssymb}   
\usepackage[all]{xy}
\usepackage{amsthm}
\usepackage{bm}
\usepackage{geometry}
\usepackage{algorithm}
\usepackage{algorithmic}
\usepackage{hyperref}
\usepackage{appendix}
\usepackage{color}
\usepackage{enumitem}
\usepackage{bm}
\usepackage{mathtools}
\usepackage{nccmath}
\usepackage{cleveref}
\usepackage{soul} 
\usepackage[foot]{amsaddr}

\DeclareMathOperator{\R}{\mathbb{R}}

\DeclareMathOperator{\id}{id}

\newcommand{\Alainrm}[1]{}

\def\bG{\bm{G}}
\def\bH{\bm{H}}
\def\bN{\bm{N}}
\def\bv{{\bm{v}}}
\def\bV{{\bm{V}}}
\def\bu{{\bm{u}}}
\def\bp{{\bm{p}}}
\def\bq{{\bm{q}}}

\def\mBk{\mathbb{B}^k}
\def\mBku{\mathbb{B}^{k+1}}

\def\vphi{\varphi}

\def\Lip{\operatorname{Lip}}
\def\prd{\text{pr}_2}

\title[GGA Model]{The Graded Group Action Framework For Sub-Riemannian Orbit Models In Shape Spaces}
\author{Thomas Pierron}\address{ENS Paris-Saclay, Centre Borelli (\texttt{thomas.pierron@ens-paris-saclay.fr})}
\author{Alain Trouvé}\address{ENS Paris-Saclay, Centre Borelli (\texttt{alain.trouve@ens-paris-saclay.fr})}

\begin{document}
\pagenumbering{Roman}
\maketitle
\begin{abstract}
  In the standard orbit model on shape analysis, a group of diffeomorphism on the ambient space equipped with a right invariant sub-riemannian metric acts on a space of shapes and induces a sub-riemannian structure on various spaces. An important example is given by the Large Deformation Diffeomorphic Metric Mapping (LDDMM) theory that has been developed initially in the context of medical imaging and image registration. However, the standard theory does not cover many interesting settings emerging in applications. We provide here an extended setting, the graded group action (GGA) framework, specifying regularity conditions to get most of the well known results on the orbit model for general groups and shape spaces equipped with a smooth structure of Banach manifold with application to multi-scale shape spaces. A specific study of the Euler-Poincaré equations inside the GCA framework leads to a uniqueness result for the momentum map trajectory lifted from shape spaces with different complexities deciphering possible benefits of over-parametrization in shooting algorithms.
\end{abstract}

\newtheorem{defi}{Definition}[section]
\newtheorem{prop}[defi]{Proposition}
\newtheorem{theo}[defi]{Theorem}
\newtheorem{lemma}[defi]{Lemma}
\newtheorem{cor}[defi]{Corollary}
\newtheorem{rem}[defi]{Remark}
\addto\captionsfrench{\def\proofname{pte}}
\addtolength{\baselineskip}{+0.2\baselineskip}
\parindent=0cm
\pagenumbering{arabic}
\tableofcontents
\newpage
\section{Introduction}
During the last decade, we have witnessed a wide development of the theory of shapes spaces driven by internal mathematical sources coming from its numerous links with other mathematical branches as infinite dimensional riemannian geometry, diffeomorphism groups, PDE, optimal control and statistics, among other. However, as illustrated by the application of shape spaces to computational anatomy, the emergence of new datasets and problems, new pratical and computational needs, are also constantly pushing the bounds of existing mathematical shape spaces frameworks to address new situations. Today, emerging imaging technologies are opening new avenues toward a multi-scale and multi-modal view of  biological tissues from the milimeter-scale of the MRI images to  sub-micrometer scale of genes expression in spatial transcriptomics \cite{BME}. In particular, an appealing route would be to consider simultaneous actions of a coarse-to-fine stacks of transformations acting on a shape shape describing a tissue at different resolutions with interaction across scales.  This introduce a substantial change in the classical framework as developped in \cite{arguillere_trelat_2017} which focused on the action of single group of diffeomorphism on a shape space. It opens up possibilities for the action of more general infinite dimensional group as free products of diffeomorphism groups, as a foundational example.\\

In this paper, we are developping (Section 2 and 3) the graded group action (GGA) framework, an extended perpective on shape spaces defined by the regular action of general groups. This encompasses the usual action of a group of diffeomorphisms of the ambiant space but allows to address with minimal efforts new situations outside this central case. It also leverage many core aspects of a ``good'' shape space theory, such as the proper definition of (sub)-riemannian structures and distances coming from right invariant metric on the group, Hamiltonian formulations, and well-defined shooting equations for computing minimizing geodesics. An illustrative example showing how recently introduced multi-scale shape spaces can be studied in the GGA framework is developed in Section 4. We then address regularity conditions for local existence and uniqueness of a $C^1$ solution of the hamiltonian flow and show that the lifted trajectories in the group satisfy a family of locally lipschitz ODEs on the group indexed by the initial value of the momentum map in line with the Euler-Poincaré equations for right invariant metrics (Section 5). Interestingly, this establishes a correspondence principle according to there initial momenta between hamiltonian evolutions on different shapes spaces associated to a common graded group structure, highlighting the potential benefits of over-parametrization in algorithms in constrast to the traditional reduction pespective (Section 6).

\section{Graded group structures}
\subsection{Admissible graded group structure}
We consider a family $G=\{G^k, k\geq 1\}$ of groups  with Banach differential structure. We  denote $\operatorname{inv}:h\mapsto h^{-1}$ the inverse mapping, and for $g\in G^k$, $L_g:h\in G^k\mapsto gh\in G^k$ and $R_g:h\in G^k\mapsto hg\in G^k$ the left and right multiplications on $G^k$. In this paper, we will assume that $G$ is an admissible graded group structure in the following sense.

\begin{defi}[Admissible graded group structure] We say that $G$ is an admissible graded group structure if the following conditions are satisfied~:
\label{def:hypgroup}
\begin{description}[leftmargin=*]
\label{HypGroup}
    \item[(G.1)] $G^{k+1}$ is a subgroup of $G^k$ with smooth inclusion.
    \item[(G.2)] For $l\geq 0$, the inverse mapping on the restriction $\operatorname{inv} : G^{k+l}\rightarrow G^k$ is $C^l$.
    \item[(G.3)] For $l\geq 0$, the induced multiplication
$$
 \begin{array}{ccl}
    G^{k+l}\times G^k & \longrightarrow & G^k \\
    (g',g) &\longmapsto & g'g\\
  \end{array}
$$
is $C^l$ and $C^{\infty}$ in the first variable $g'$ for $g$ fixed.
    \item[(G.4)] For $l\geq 0$, the induced left infinitesimal action 
    $$
 \begin{array}{ccl}
    T_eG^{k+l}\times G^k & \longrightarrow & TG^k \\
    (u,g) &\longmapsto & u\cdot g = \partial_{g'} (g'g)_{|g'=e}(u)=T_eR_g(u) \in T_gG^k \\
  \end{array}
$$ is a $C^l$ mapping, and $C^{\infty}$ with regards to the first variable. 
\item[(G.5)] The induced right infinitesimal action
$$
\begin{array}{ccl}
    G^{k+1}\times T_eG^k & \longrightarrow & TG^k \\
    (g,u) &\longmapsto & T_eL_g(u) \\
  \end{array}
$$
is $C^1$
\end{description}
\end{defi}

\begin{rem}
    This definition is similar to the concept of ILB-Lie groups of Omori \cite{hideki1997infinite}, exept the definition of Omori imposes that $G^{k+1}$ is a dense subgroup of $G^k$ \cite[Thm 3.7]{hideki1997infinite}. Moreover, in the context of Large deformations, we are not interested here in the properties of the limit $\bigcap_k G^k$, and we will define metrics directly on $G^k$.
\end{rem}

We have the following immediate consequences of properties (G.3) and (G.5)~:
\begin{prop}
    \label{smooth_op}
    For any $k\geq 1$, $G^k$ is a topological group. Morever,  we have 
    \begin{enumerate}
        \item $R_g : G^k\rightarrow G^k$ is  $C^\infty$ for any $g \in G^k$
        \item $L_g : G^k\rightarrow G^k$ is $C^l$ for any $g\in G^{k+l}$
        \item The left multiplication $g\in G^{k+1}\mapsto T_e L_g \in \mathcal{L}(T_eG^k,T_gG^k)$ is locally-lipschitz, in the sense it is locally-lipschitz in any chart.
    \end{enumerate} 
\end{prop}
\begin{proof}
    The two first points are immediate with (G.3). Let us prove the last point. Suppose $\mathbb{B}^{k}$ (resp. $\mathbb{B}^{k+1}$) is a Banach space that models the group $G^{k}$ (resp. $G^{k+1}$). Let $\pi : TG^k \to G^{k+1}$ denote the pullback of the vector bundle $TG^k\to G^k$ by the inclusion $i:G^{k+1}\hookrightarrow G^k$. Let $U\subset G^{k+1}$ be a local chart of $G^{k+1}$ and $x\in U\to \widehat{T_eL}_x\in \mathcal{L}(T_eG^k,\mathbb{B}^k)$ be a local representation of $g\mapsto T_eL_g$ in local charts of the vector bundle $\pi : TG^k \to G^{k+1}$. By (G.5) $(x,u)\in U\times T_eG^k\mapsto \widehat{T_eL}_x(u) = \widehat{T_eL}(x,u)$ is $C^1$. Therefore, for $x\in U$, there exists $K,\delta>0$ and $U_0$ open convex neighborhood of $x$ in $\mathbb{B}^{k+1}$, such that for any $y\in U_0$ and $u\in B_{\mathbb{B}^k}(0,\delta)$, we have
    $$
    |\partial_1(\widehat{T_eL})(x,u) - \partial_1(\widehat{T_eL})(y,u)|_{\mathcal{L}(\mathbb{B}^{k+1},\mathbb{B}^k)}<K
    $$
    Therefore we get, for $y,y'\in U_0$ and $u\in B_{\mathbb{B}^k}(0,\delta)$
    \begin{align*}
        |\widehat{T_eL}(y,u) - \widehat{T_eL}(y',u)|_{\mathbb{B}^k} &\leq \int_0^1\left|\partial_1(\widehat{T_eL})\left((1-t)y+ty',u\right)\left(y'-y\right)\right|_{\mathbb{B}^k}dt \\
        &\leq \int_0^1\left|\partial_1(\widehat{T_eL})\left((1-t)y+ty',\frac{u\delta}{|u|}\right)\left(y'-y\right)\right|_{\mathbb{B}^k} \frac{|u|}{\delta}dt \\
        &\leq \int_0^1 K  |y'-y|_{\mathbb{B}^{k+1}} \frac{|u|}{\delta}dt = K  |y'-y|_{\mathbb{B}^{k+1}} \frac{|u|}{\delta}
    \end{align*}
Thus we finally get 
$$
|\widehat{T_eL}(y,u) - \widehat{T_eL}(y',u)|_{\mathcal{L}(T_eG^k,\mathbb{B}^k)} < \frac{K}{\delta} |y'-y|_{\mathbb{B}^{k+1}} 
$$ and $g\mapsto T_eL_g$ is locally lipschitz.
\end{proof}

\subsection{Adjoint representation}
The groups $G^k$ are not in general Lie groups, but we can define an equivalent for the adjoint representations..
Let $g \in G^{k+1}$ and let $$\mbox{int}_g : \left|
  \begin{array}{ccl}
    G^k & \longrightarrow & G^k \\
    h &\longmapsto & ghg^{-1} \\
  \end{array}
\right.$$ be the conjugation by $g$

\begin{defi} For all $g\in G^{k+1}$, the automorphism $\operatorname{int}_g$ is $C^1$. We define its derivative :
\begin{equation}
    \operatorname{Ad}_g : \left|
  \begin{array}{ccl}
    T_eG^k & \longrightarrow & T_eG^k \\
    v &\longmapsto &  \operatorname{Ad}_g(v) = T_{e}\operatorname{int}_g(v)\\
  \end{array}
\right.
\end{equation}
\end{defi}
\begin{proof}
Let $g\in G^{k+1}$. We have for $h\in G^k$, $\operatorname{int}_g(h) = R_{g^-1}\circ L_{g}(h)$, and by hypothesis, $L_g$ is $C^1$ on $G^k$ and $R_{g^{-1}}$ is also $C^1$ on $G_k$.
\end{proof}

\subsection{Absolutely continuous curves and differentiable structure}
Now we prove some regularity result on the groups, i.e. we prove there exist a unique global flow associated to a right-invariant vector field with some conditions. We start by giving the definition of absolutely continuous curves with $L^p$ derivative in the groups (or more generally in Banach manifolds) following the work from Glöckner \cite{Glo}. 

Let $\mathbb{B}^k$ a Banach space that models the group $G^k$. Let $p\in [1,+\infty]$ and $a<b\in\R$. We define the vector space $AC_{L^p}([a,b],\mathbb{B}^k)$ of continuous curves $\eta : [a,b] \rightarrow \mathbb{B}^k$ such that there exists $\gamma \in L^p([a,b],\mathbb{B}^k)$ verifying for any $t\in[a,b]$
\begin{equation}
    \eta(t) = \eta(a) + \int_{a}^t \gamma(t) dt 
\end{equation}
This is equivalent to saying that $\eta$ is almost everywhere differentiable with $\eta'\in L^p([a,b],\mathbb{B}^k)$. We introduce on the space $AC_{L^p}([a,b],\mathbb{B}^k)$ the norm $|\cdot|_{AC_{L^p}}$ given by :
$$
|\eta|_{AC_{L^p}} = |\eta(a)|_{\mathbb{B}^k} + |\eta'|_{L^p}
$$
Then $\left(AC_{L^p}([a,b],\mathbb{B}^k),|\cdot|_{AC_{L^p}}\right)$ is a Banach space and we have the continuous inclusion :
$$
AC_{L^p}([a,b],\mathbb{B}^k)\hookrightarrow C([a,b],\mathbb{B}^k)\times L^p([a,b],\mathbb{B}^k)
$$

Now let $I\subset \R$ an interval. We define as $AC_{L^p}(I,G^k)$ as the set of curves $\eta : I \rightarrow G^k$, such that for any local charts $(U,\varphi)$ and  any $a<b$ such that $\eta([a,b])\subset U$, the curve
$$
\varphi\circ\eta : [a,b]\rightarrow \mathbb{B}^k
$$
is in $AC_{L^p}([a,b],\mathbb{B}^k)$. 

We turn now to the definition of a smooth manifold structure on the space of absolutely continuous curves on the full group $G^k$. Glöckner \cite{Glo} proved that if $G$ is a Banach Lie group, then $AC_{L^p}(I,G)$ is also a Banach Lie group. More generally in \cite{Schmeding2016ManifoldsOA}, the author proved that if $M$ is a Banach manifold equipped with some strong riemannian metric, then $AC_{L^p}(I,M)$ also remains a Banach manifold. Following an older work by Krikorian \cite{Krikorian-1972}, that was continued in \cite{golinski2021regulated} for regulated curves, it is however possible to put a differentiable structure on such a space, without any other hypotheses on the manifold $M$ (see appendix \ref{app:A} for a self-contained exposition of the Banach manifold structure). Based on this later construction 
we get the following result.

\begin{prop}
\label{man_acg} 
Assume $I=[a,b]$ with $a<b\in \mathbb{R}$. 
\begin{enumerate}
    \item \label{prop:BM}
The space $AC_{L^p}(I,G^k)$ is a Banach manifold.

\Alainrm{Furthermore, for $a=(a_i)_{0\leq i\leq n}$ such that $0=a_0<\ldots<a_n=1$ and $(U,\varphi)=(U_i,\varphi_i)_{1\leq i\leq n}$ a family of open charts on $\mathcal{M}$, the set $AC_{L^p}(a;U)$ defined as
$$
AC_{L^p}(a;U) = \{ \eta\in C(I,G^k)\ |\ \eta(I_i)\subset U \text{ and }\varphi_i\circ\eta_{|I_i}\in AC_{L^p}(I_i,\mathbb{B}),\ 1\leq i\leq n\ \}
$$ 
is an open submanifold of $AC_{L^p}(I,TG^k)$. The mapping $\Phi : AC_{L^p}(a;U) \rightarrow \prod_{i=1}^n AC_{L^p}(I_i,\mathbb{B})$ defined as
$$
\Phi(\eta) = (\varphi_i\circ\eta_{|I_i})_{1\leq i \leq n}
$$ is a diffeomorphism.}
\item \label{prop:ev} For $t\in I$, the evaluation
$$
\operatorname{ev}_{t} : \left\{\begin{array}{ccl}
    AC_{L^p}(I,G^k) & \longrightarrow & G^k  \\
    \eta & \longmapsto & \eta(t) 
\end{array} \right.
$$ is smooth.
\item \label{prop:tan} The vector bundle $AC_{L^p}(I,TG^k)\rightarrow AC_{L^p}(I,G^k)$ can be taken as the tangent bundle. For $g\in AC_{L^p}(I,G^k)$, we therefore have
$$
T_g AC_{L^p}(I,G^k) = AC_{L^p}(I\leftarrow g^*TG^k)
$$
where $AC_{L^p}(I\leftarrow g^*TG^k)=\{\gamma \in AC_{L^p}(I,TG^k), \ \gamma(t) \in T_{g(t)}G^k, \ \forall t \in I\}$
\end{enumerate}
\end{prop}
\begin{proof}
    The proof follows with some adaptations the arguments of the proof of theorem (3C) and the results of (§4) of \cite{Krikorian-1972}.
    A detailed proof is given in appendix A where \eqref{prop:BM} comes from Prop \ref{man_ac}, \eqref{prop:ev} comes from Prop \ref{propA:ev} and \eqref{prop:tan} from Proposition \ref{propA:tan}.
    \Alainrm{
    In this setting, we need to verify $AC$ satisfies the conditions (mm1-4) of a manifold model defined in (§1), which is a consequence of 3.16 and lemma 3.27 in \cite{Glo}. The only technical point is to prove that if $[\alpha',\beta']\subset [\alpha,\beta]$ are two subintervals of $I$, then the restriction mapping $AC_{L^p}([\alpha,\beta],\mathbb{B}^k)\rightarrow AC_{L^p}([\alpha',\beta'],\mathbb{B}^k), \eta \mapsto \eta_{|[\alpha',\beta']}$ is a linear continuous submersion, meaning it has a linear continuous right inverse. In other words, since the restriction is linear and continuous (see \cite{Glo}), we just need to prove that any $\tilde {\eta} \in AC_{L^p}([\alpha',\beta'],\mathbb{B}^k)$ can be extended to a curve $\eta\in AC_{L^p}([\alpha,\beta],\mathbb{B}^k)$ in a linear and continuous way, which can be done with constant parts. We will give a detailed proof of this result in the appendix.}
\end{proof}

The rest of this section is devoted to the study of absolutely continuous curves in $G^k$. We start by showing the existence of an absolutly continuous lift in $G^k$ for any integrable curve in the tangent space $T_eG^{k+1}$.
\begin{prop}[Evolution equation in $G^k$]
\label{reg_groups}
    Let $I\subset \R$ an interval, $t_0\in I$, and $u\in L^p(I,T_eG^{k+1})$. Then the ordinary differentiable equation :
    \begin{equation}
    \label{evol_group}
        \left\{
        \begin{array}{l}
             \dot{g}_t = u_t\cdot g_t = T_eR_{g_t}(u_t) \\
             g_{t_0} = e
        \end{array}
        \right.
    \end{equation}
    has one unique global (i.e. defined on $I$) solution $g\in AC_{L^p}(I,G^k)$.
    
    Furthermore, if $u\in L^p(I,T_eG^{k+n})$ for $n\geq 1$, then for all $t\in I$, the left translation $L_{g_t}:G^k\rightarrow G^k$ is $C^n$.
\end{prop}
\begin{proof}

The first step is to prove local existence of solutions, using classic Picard-Lindelof theorem. Let $(U,\varphi)$ be a local chart around $e$ in $G^k$ with $\varphi : U \rightarrow \mBk$ where $\mBk$ is the modelling Banach space for $G^k$. We denote $V=\vphi(U)\subset \mBk$ and $d_g\vphi\doteq \prd\circ T_g\vphi$ for any $g\in U$ where $\prd:V\times\mBk\to \mBk$ the canonical projection on the second argument. We consider $f:V\times \mBku\to \mBk$ such that
  $$f(y,v)=d_{g}\vphi(u\cdot g)\text{ with }(u,g)=((d_e\vphi)^{-1}(v),\varphi^{-1}(y))\,.$$
  From (G.4), $f$ is $C^{1}$. Since $(y,v)\mapsto \partial_y f(y,v)$ is continuous on $V\times \mBku$ and linear in $v$, we can assume (up to the restriction of $V$ to a smaller open set) that there exists $K>0$ such that $\|\partial_y f(y,v)\|_{\mathcal{L}(\mBk,\mBk)}\leq K|v|_{\mBku}$ for any $(y,v)\in V\times \mBku$.  Hence  the mapping $F: I \times \vphi(U)  \longrightarrow  \mBk$ defined by $F(t,y)=f(y,v_t)$ where $v_t=d_e\vphi(u_t)$ is such that
  $$|F(t,y)-F(t,y')|_{\mBk}\leq K|v_t|_{\mBku}|y-y'|_{\mBk}$$
  so that $F(t)\doteq F(t,.)\in \Lip(V,\mBk)$ and, since $t\mapsto v_t$ is $L^p(I,\mBku)$ for $u\in L^p(I,T_eG^{k+1})$, $t\mapsto F(t)\in L^1_{loc}(I,\Lip(V,\mBk))$.
  Therefore by \cite[Theorem C.6]{Younes2019}, for any $t_0 \in I$, the equation \eqref{evol_group} that is locally equivalent to $\dot{y}_t = F(t,y_t)$ has one unique local solution in $G^{k}$ defined on $I_{t_0}\subset I$.  In the case where $u\in L^p(I,T_eG^{k+n})$ for some $n\geq 1$, we even have that $t\mapsto F(t)\in L^1_{loc}(I,C^n_b(V,\mathbb{B}^k))$. Therefore, by \cite[Theorem C.15 and C.18]{Younes2019} for $t\in I_{t_0}$, the mapping $y_t:x\in V \mapsto y_t(x)\in \mathbb{B}^k$ is $C^n$, where $y_t(x)$ is solution of the equation $\dot{y}_t = F(t,y_t)$ with initial point $y_{t_0}=x$. By uniqueness of the solution, this corresponds in $G^k$ to a $C^n$ mapping $g\in U \mapsto g_t g=L_{g_t}(g)\in G^k$ so that $L_{g_t}$ is locally $C^n$ on $U$. Since $R_h(U)=Uh=R_{h^{-1}}^{-1}(U)$ is an open neighborhood of $h\in G^k$ and since for $y \in R_h(U)$, $L_{g_t}(y)= R_h\circ L_{g_t}\circ R_{h^{-1}}(y)$ where $R_h$ and $R_{h^{-1}}$ are smooth mappings (see (G.3)), we get that $L_{g_t}$ is $C^n$ on $G^k$ for all $t\in I_{t_0}$. 
  
  Now, following the proof from \cite{kriegl1997convenient,Glo}, we prove that the solution is globally defined on $I$.  
  It is enough to consider the case where $I=[a,b]$ is compact. We get from the local existence result that there exists $a=a_1<a_2<\ldots <a_n=b$ such that $I=[a_1,a_n]$ and for any $1\leq i<n$ a solution $g_i \in AC_{L^p}([a_i,a_{i+1}],G^k)$ of 
  $$
 \left\{
        \begin{array}{l}
             \dot{g}_i(t) = u(t)\cdot g_i(t)  \\
             g_i(a_i) = e\,.
        \end{array}
        \right.
  $$
  We can now define a global solution $g(t) = g_{i}(t)g_{i-1}(a_{i})\ldots g_1(a_2)=R_{g_{i-1}(a_{i})\ldots g_1(a_2)}(g_{i}(t))$ for $a_i\leq t \leq a_{i+1}$ which is in $AC_{L^p}(I,G^k)$ as $R_{g_{i-1}(a_{i})\ldots g_1(a_2)}$ is smooth and each $g_i$ is $AC_{L^p}$.
\end{proof}
    From the previous proposition, we can define the evolution map for any time dependent vector field $u\in L^p(I,T_eG^{k+1})$.
\begin{defi}[Evolution map]
    We denote  $\operatorname{Evol}_{G^k} : L^p(I,T_eG^{k+1}) \rightarrow AC_{L^p}(I,G^k)$ the application associating to any time-dependent vector field $u\in L^p(I,T_eG^{k+1})$, the solution $g \in AC_{L^p}(I,G^k)$ of the ODE \eqref{evol_group}. 
\end{defi}

We can prove that the evolution map has some regularity, depending on the space where $u$ lives$~$:

\begin{prop}
\label{Evol_in_G}
Suppose $I$ is a compact interval.
\begin{enumerate}
    \item 
The restriction $\operatorname{Evol}_{G^k} :  L^p(I,T_eG^{k+1+l}) \rightarrow AC_{L^p}(I, G^k)$ with $p\geq 1$ is $C^l$. 
\item For $u,\delta u \in L^p(I,T_eG^{k+1+l})$, its derivative $\delta g = T_u \operatorname{Evol}_{G^k}(\delta u) \in AC_{L^p}(I\leftarrow g^*TG^k)$ is the unique solution of the linear Cauchy problem :
\begin{equation}
\label{linq}
    \delta \dot{g}(t) = \partial_g (u(t)\cdot g)_{|g=g(t)}\delta g(t) + \partial_u(u\cdot g(t))_{|u=u(t)}\delta u(t), \ \delta g(0)=0
\end{equation}
where $g(t)=\operatorname{Evol}_{G^k}(u)(t)$.
\label{evol_reg}
\end{enumerate}
\end{prop}
\begin{proof} 
  \def\LpG{L^p(I,T_eG^{k+l+1})}
  \def\Evol{\operatorname{Evol}_{G^k}}
  \def\Bk{\mathbb{B}^k}
  \def\Bklp{T_eG^{k+l+1}}
  \def\yev{y_{\text{ev}}}
  We assume $I=[0,1]$ and we prove the result locally. Let $u_0\in \LpG$, $g_0=\Evol(u_0)$ and consider a chart $(\mathcal{U}=AC_{L^p}(a;U),\Phi)$ around $g_0$ as introduced in appendix \ref{man_ac} where $a=(a_i)_{1\leq i\leq n}$, $U=(U_i)_{1\leq i\leq n}$ and $\phi=(\varphi_i)_{1\leq i\leq n}$ are such that $0=a_0<a_1<\ldots <a_n=1$, and $(U_1,\varphi_1),\ldots ,(U_{n},\varphi_n)$ is a  collection of charts on $G^k$. In the sequel we denote $I_i=[a_{i-1},a_i]$ and $V_i=\varphi_i(U_i)$ for $1\leq i\leq n$. 

  Working in local coordinates, consider for $1\leq i\leq n$ the mapping $V_i\times \Bklp\to \Bk$ defined by $(y^i,u)\mapsto u\cdot y^i\doteq d_g\varphi_i(u\cdot g)= d_e(\varphi_i\circ R_{g})(u)$ with $g=\varphi_i^{-1}(y^i)$. From $(G.4)$ we get that this mapping is $C^{l+1}$ and therefore induced a mapping $AC_{L^p}(I_i,V_i)\times L^p(I_i,\Bklp)\to L^p(I_i,\Bk)$ defined by  $(y^i,u) \mapsto u\cdot y^i\doteq (t\mapsto u(t)\cdot y^t(t))$ which is $C^l$ \cite[prop 2.3]{Glo} so that we get eventually a $C^l$ mapping $(\prod_{i=1}^nAC_{L^p}(I_i,V_i))\times \LpG\to \prod_{i=1}^n L^p(I_i,\Bk)$ defined by  $(y,u)\mapsto u\cdot y\doteq (u.y^i)_{1\leq i\leq n}$ for $y=(y^i)_{1\leq i\leq n}$.\\

We consider now the $C^l$ mapping 
  $$
  C :\left\{
  \begin{array}{ccc}
      \prod_{i=1}^{n}AC_{L^p}(I_i,V_i)\times L^p(I,T_eG^{k+1+l})& \rightarrow & \mathbb{B}^k\times \prod_{i=1}^{n} L^p(I_i,\mathbb{B}^k)\times (\Bk)^{(n-1)}  \\
    \left(y,u\right)  &\mapsto & \left(y^1(0), \dot{y}-u\cdot y,\sigma(y)\right)
  \end{array}
  \right.
  $$
  where $\sigma(y)=(\varphi_{i+1}\circ (\varphi_i)^{-1}\circ y^i(a_i)-y^{i+1}(a_i))_{1\leq i<n}$ is the smooth mapping introduced in appendix  \ref{man_ac} checking for the continuity conditions at the boundaries of the segments $I_i$ so that if $\yev(u)=\Phi(\Evol(u))$  for $u\in L^p(I,T_eG^{k+1+l})$, then $\yev(u)$ verifies~:
  $$
    C(\yev(u),u)=(e_{G^k},0,0)
$$
However,
  \begin{equation*} 
         \partial_{y} C\left(y,u\right)\delta y=
      \left(\delta y^1(0),\delta\dot{y} - \partial_{y} (u\cdot y)\delta y, d_y\sigma(\delta y)\right)
  \end{equation*}
  and for any $(\delta q_0,\delta w, \delta \sigma)\in \mathbb{B}^k\times \prod_{i=1}^{n} L^p(I_i,\mathbb{B}^k)\times(\Bk)^{(n-1)}$, the equation
  $$\partial_yC(y,u) \delta y = (\delta q_0,\delta w, \delta\sigma)$$ is a linear Cauchy problem on each segment $I_i$ with boundary conditions induced by $\delta q_0$ and $\delta\sigma$ at time $(a_i)_{0\leq i\leq n}$ that admits a unique global solution $\delta y \in  \prod_{i=1}^{n}AC_{L^p}(I_i,V_i)$. Thus $\partial_y C(y,u) : \prod_{i=1}^{n}AC_{L^p}(I_i,V_i)\rightarrow \mathbb{B}^k\times \prod_{i=1}^{n} L^p(I_i,\mathbb{B}^k)\times (\Bk)^{(n-1)}$ is a Banach isomorphism and by implicit function theorem, there exists an open neighborhood $W_0\subset L^p(I,T_eG^{k+1+l})$ of $u_0$ such that $\yev$ and $\Evol:W_0\rightarrow AC_{L^p}(a;U)$ are $C^l$.\\

Moreover, $\delta y =T_u\yev(\delta u)$ is solution of
  $\partial_yC(y,u)\delta y+\partial_uC(y,u)\delta u=$
  i.e. $\delta y^{1}(0)=0,\ d_y\sigma(\delta y)=0$ and
  $$\delta \dot{y} =\partial_u(u\cdot y)\delta u +\partial_y(u\cdot y)\delta y=0$$
  which is again  a linear Cauchy problem on each segment $I_i$ with boundary conditions induced by $\delta y^1(0)=0$ and $\delta\sigma = 0$ at  $t=a_i$ for $0\leq i< n$ and admitting a unique global. Thus, for $g=\Evol(u)$, we get that $\delta g$ is solution of \eqref{linq}.
\end{proof}

\subsection{The example of half-Lie groups}
The category of Banach half-Lie groups gives us many examples of families of groups satisfying the (G.1-5) conditions. Riemannian geometries on such spaces were recently studied by Bauer, Harms and Michor in \cite{bauer2023regularity} :
\begin{defi}
    A Banach half-Lie group is a topological group with smooth Banach manifold structure, such that the left multiplication $g\mapsto gg'$ is smooth.
\end{defi}
For the following, we will consider the half-Lie groups to carry a right-invariant local addition as defined in \cite[Definition 3.2]{bauer2023regularity}, which will allow all regularity properties.

Given $G$ a Banach half-Lie group, we wish to define a family $(G^k)$ satisfying conditions (G.1-5). For $k\in \mathbb{N}$, we define as $G^k$ the set of $g\in G^k$ such that $L_g : G\rightarrow G$ and $L_{g^{-1}}: G\rightarrow G$ are $C^k$. We then have the following proposition

\begin{prop}
\label{link_halfLie}
    The family of groups $\{G^k, k\in \mathbb{N}^*\}$ satisfies the conditions (G.1-5)
\end{prop}
\begin{proof}
    The proof is mostly included in the work of Bauer, Harms and Michor in \cite[Theorem 3.4]{bauer2023regularity}, using the identification through the smooth diffeomorphism 
    $$
    \begin{array}{ccl}
    G^k & \longrightarrow & \operatorname{Diff}_{C^k}(G)^G \\
    g &\longmapsto &  L_g\\
  \end{array}
    $$
    where $\operatorname{Diff}_{C^k}(G)^G$ is the space of right-invariant $C^k$-diffeomorphisms on $G$, i.e. diffeomorphisms that commute with $L_g$ for all $g\in G^k$.
\end{proof}

\emph{Examples :} The authors in \cite{Marquis2018HALFLIEG} introduce a wide variety of examples of half-Lie groups, and in particular the groups of diffeomorphisms of finite regularity on a compact manifold. The specific case of diffeomorphisms group of Sobolev regularity was mainly discussed in context of fluid dynamics \cite{ebin_marsden} and in shape analysis \cite{arguillere_trelat_2017,bauer2023regularity,BaBruMichor,bruveris2017completeness}. 

In this paper, we will focus on the group of diffeomorphisms of finite regularity that are equivalent to identity at infinity :
$$
\operatorname{Diff}_{C_0^k}(\R^d) = \{\id + g, g\in C_0^k(\R^d,\R^d), \det(I_d+dg)>0\}
$$
This group is an open subset of the Banach space $C_0^k(\R^d,\R^d)$. The mappings $\varphi,\psi\mapsto\varphi\circ\psi$ and $\varphi\circ\varphi^{-1}$ are continuous (by using Faà di Bruno's formula for the computation of the derivatives of the compositions of mappings, see \cite{bauer2023regularity} for example), and for $\psi\in\operatorname{Diff}_{C_0^k}(\R^d)$, the mapping $\varphi\mapsto\varphi\circ\psi$ is linear and continuous, hence smooth. Therefore the group $\operatorname{Diff}_{C_0^k}(\R^d)$ is a Banach right half-Lie group. In this case, for $l\geq0$, the space of $C^l$ elements of $\operatorname{Diff}_{C_0^k}(\R^d)$ is simply $(\operatorname{Diff}_{C_0^k}(\R^d))^l = \operatorname{Diff}_{C_0^{k+l}}(\R^d)$ such that proposition \ref{link_halfLie} holds.

\section{Sub-Riemannian geometry on Banach manifolds induced by right-invariant metrics}
\label{sec:sub-riem}
\subsection{Definition of the shape space}
In this section we introduce a Banach manifold $\mathcal{Q}$. We want to define strong sub-Riemannian structure on $\mathcal{Q}$ induced by the action of a family of groups $G^k$ equipped with a right-invariant sub-Riemannian metric. Let $k_0\geq 0$ and $\{G^k, k\in \mathbb{N}\}$ a family of topological Banach groups satisfying the conditions (G.1-5) introduced in definition \ref{def:hypgroup}. Let $V$ a Hilbert space continuously embedded in $T_eG^{k_0+2}$. We suppose $G^{k_0}$ acts on $\mathcal{Q}$ and we denote A the action
$$A :\left|
  \begin{array}{ccl}
    G^{k_0}\times \mathcal{Q} & \longrightarrow & \mathcal{Q} \\
    (g, \ q) &\longmapsto & g \cdot q \\
  \end{array}
\right.
$$
We suppose that the following conditions are satisfied :

\begin{description}[leftmargin=*]
    \label{HypActions}
\item[(S.1)] \emph{Continuity of the action} : $A : (g,q)\mapsto g\cdot q$ is continuous
\item[(S.2)] \emph{Infinitesimal action} : For all $q \in \mathcal{Q}$, the mapping $A_q = A(\cdot,q) : g \mapsto g\cdot q$ is $C^{\infty}$, and we denote $\xi_q = \xi(\cdot,q) = \partial_g A(g,q) _{|g=e}$ its continuous differential in $e$.
\item[(S.3)] \emph{Regularity of the action} : For $l>0$, the mappings 
$$
\begin{array}{ccc}
A :\left|
  \begin{array}{ccl}
    G^{k_0+l}\times \mathcal{Q} & \longrightarrow & \mathcal{Q} \\
    (g, \ q) &\longmapsto & g \cdot q \\
  \end{array}
\right.
     & \mbox{ and } &
     \xi : \left|
  \begin{array}{ccl}
    T_eG^{k_0+l}\times \mathcal{Q} & \longrightarrow & T\mathcal{Q} \\
    (u, \ q) &\longmapsto & \xi_q(u) = u\cdot q \\
  \end{array}
\right.
\end{array}
$$
are $C^l$.
\end{description}

\begin{rem}
    Since $V\hookrightarrow T_eG^{k_0+2}$, condition (S.3)  implies that the mapping $\xi : V\times \mathcal{Q}\to T\mathcal{Q}$ is also a $C^1$-vector bundle morphism in the sense of \cite{lang2001fundamentals}, meaning that $q\in\mathcal{Q}\mapsto \xi_q \in \mathcal{L}(V,T\mathcal{Q})$ is $C^1$. Furthermore its derivative is locally lipschitz. If $G$ is a finite dimension Lie group, then $\xi$ is even a $C^2$-vector bundle morphism, but this property does not hold in infinite dimension. 
\end{rem}

In the following, we consider the general framework of a graded group structure acting on a Banach manifold that encompasses both (G.1-5) and (S.1-3) conditions :
\begin{description}
\item[(GGA)]
 $\{G^k, k\in \mathbb{N}\}$ is an admissible graded group structure satisfying the conditions (G.1-5), $\mathcal{Q}$  is a Banach manifold playing the role of a shape space and $k_0\geq 0$ is an integer such that $G^{k_0}$ acts on $\mathcal{Q}$ and satisfies conditions (S.1-3) with $V$ a Hilbert space continuously embedded in $T_eG^{k_0+2}$.
\end{description}

\begin{rem}
    \label{GGAonG}
    In particular, the multiplication of the group $G^{k_0}$ induces a left action that verifies $(S.1-3)$ conditions. Therefore, the family of groups $\{G^k, k\in\mathbb{N}\}$ endowed with its natural left action verifies the (GGA) conditions, with $\mathcal{Q}=G^{k_0}$ and $V\hookrightarrow T_eG^{k_0+2}$. In this case, the infinitesimal action is simply the derivative of the right multiplication :
    $$
    \xi_g(u) = T_eR_g(u) = u\cdot g
    $$
\end{rem}

\subsection{Horizontal curves and distance}
In this part, we will denote by $I$ the closed interval $[0,1]$, and we suppose the hypotheses of the (GGA) framework are satisfied.

\begin{defi}[Horizontal curves in $\mathcal{Q}$]
    An absolutely continuous curve $q : I \rightarrow \mathcal{Q}$ is said to be horizontal if there exists a continuous lift $t\mapsto u(t)\in V$ such that
    $$
    \forall t \in I, \ \dot{q}(t)=\xi_{q(t)}(u(t))
    $$
    We call an horizontal system such a couple $(q,u) \in AC_{L^1}(I,\mathcal{Q})\times L^1(I,V)$.
\end{defi}

\begin{rem}If $q:I\rightarrow \mathcal{Q}$ is a horizontal curve in $\mathcal{Q}$, there can exist two different controls $u_1\neq u_2$ such that 
$$
\dot{q}=\xi_q(u_1)=\xi_q(u_2)
$$
Also, if $u\in L^1(I,V)$, we can always integrate $u$ in $\mathcal{Q}$, meaning there always exists an horizontal curve $q:I\rightarrow \mathcal{Q}$ such that $\dot{q}(t)=\xi_{q(t)}(u(t))$. 
\end{rem}

Similarly to propositions \ref{reg_groups} and \ref{Evol_in_G}, we can define an evolution map in $\mathcal{Q}$ by solving the corresponding ordinary differential equation.
\begin{theo}[Evolution map in $\mathcal{Q}$ and regularity]
    Let $q_0\in \mathcal{Q}$, and $u \in L^2(I,V)$. There exists a unique $q\in AC_{L^2}(I,\mathcal{Q})$ such that $q(0)=q_0$ and
    \begin{equation}
    \label{eq_in_Q}
        \dot{q}(t) = \xi_{q(t)}\big(u(t)\big) \ \mbox{a.e.}
    \end{equation}
    Furthermore the evolution map in $\mathcal{Q}$
    $$
    \operatorname{Evol}_{\mathcal{Q}} : u\in L^2(I,V) \mapsto q^u \in AC_{L^2}(I,\mathcal{Q})
    $$
    where $q^u$ is solution of equation \eqref{eq_in_Q}, is $C^1$. For $u,\delta u \in  L^2(I,V)$, its derivative $\delta q = T_u \operatorname{Evol}_{\mathcal{Q}}\cdot \delta u$ is solution of the linear Cauchy problem

\begin{equation}
\label{linq2}
    \delta \dot{q}(t) = \partial_q \left(\xi_{q}(u(t))\right)_{|q=q^u(t)}\delta q(t) + \partial_u(\xi_{q^u(t)}(u))_{|u=u(t)}\delta u(t), \ \delta q(0)=0
\end{equation}
\end{theo}

\begin{proof}
    This follows from Proposition \ref{Evol_in_G}. As $u\in L^2(I,V)$, and $\xi:V\times\mathcal{Q}\rightarrow T\mathcal{Q}$ is $C^2$ (S.3), there exists by Picard-Lindelof one unique maximal (not necessarily global) solution $q$ of equation \eqref{eq_in_Q}. There also exists by proposition \ref{Evol_in_G} a unique $g = \operatorname{Evol}_{G^{k_0}}(u) \in AC_{L^2}(I,G^{k_0})$ such that 
    $$
    \dot{g}(t) = u(t)\cdot g(t)
    $$
    It is immediate to verify that $t\mapsto A(g(t),q_0)$ is also solution to equation \eqref{eq_in_Q}, and thus $q(t)=A(g(t),q_0)$. Therefore we have $\operatorname{Evol}_{\mathcal{Q}} = \tilde{A}_{q_0}\circ\operatorname{Evol}_{G^{k_0}}$, where $\tilde{A}_{q_0}$ is given by :
    $$
    \tilde{A}_{q_0} : \left\{\begin{array}{ccl}
         AC_{L^2}(I,G^{k_0}) & \longrightarrow & AC_{L^2}(I,\mathcal{Q})  \\
         g & \longmapsto & [t\mapsto A_{q_0}(g(t))] 
    \end{array}\right.
    $$
    By (S.2), the mapping $A_{q_0}$ is smooth, and therefore the induced mapping $\tilde{A}_{q_0}$ is also smooth \cite[lemma 3.27]{Glo}, with derivative $T_g\tilde{A}_{q_0} : AC_{L^2}(I,TG^{k_0})\rightarrow AC_{L^2}(I,\mathcal{Q})$, $g \in AC_{L^2}(I,G^{k_0})$, given by $$T_g\tilde{A}_{q_0} : \delta g \mapsto [t\mapsto T_{g(t)}A_{q_0}(\delta g(t)) = \xi_{g(t)\cdot q_0}\circ (T_eR_{g(t)})^{-1}(\delta g(t))]  $$
    Now as $L^2(I,V)\hookrightarrow L^2(I,T_e G^{k_0+2})$ smoothly, then the evolution mapping in $\mathcal{Q}$ is also $C^1$ by composition. For $u, \delta u \in L^2(I,V)$, and $\delta q = T_u \operatorname{Evol}_{\mathcal{Q}}(\delta u) = T_{\operatorname{Evol}_{G^{k_0}}(u)}\Tilde{A}_{q_0}\circ T_u\operatorname{Evol}_{G^{k_0}}(\delta u)$ we thus have
    $$
    \delta q(t) = \xi_{q^u(t)}\circ (T_eR_{g(t)})^{-1}(\delta g(t))
    $$ where $g=\operatorname{Evol}_{G^{k_0}}(u)$, and $\delta g = T_u\operatorname{Evol}_{G^{k_0}}(\delta u)$ satisfies equation \eqref{linq}. By derivation of this equation in charts (or adapting the proof of proposition \ref{Evol_in_G} and introducing the mapping $ q, u \mapsto  \dot{q} - \xi_qu$) we get that $\delta q$ is solution of \eqref{linq2}.
\end{proof}

In the following, we will denote 
$$\operatorname{Hor}_{q_0}(I) = \{(q,u)\in AC_{L^2}(I,\mathcal{Q})\times L^2(I,V), \ q=\operatorname{Evol}_{\mathcal{Q}}(u) \mbox{ and } q(0)=q_0\}$$ the space of all horizontal systems with starting point $q_0$. It is in general not a submanifold of $AC_{L^2}(I,\mathcal{Q})\times L^2(I,V)$ but it's in bijection with $L^2(I,V)$.
Now we define the energy and length of an horizontal system :
\begin{defi}[Length and energy]
    Let $(q,u):I\rightarrow\mathcal{Q}\times V$ a horizontal system. We define its length and energy respectively by
    $$
\begin{array}{ccc}
     L(q,u) = \int_I |u(t)|_V dt
     & \mbox{ and } &
    E(q,u)=\frac{1}{2}\int_I |u(t)|_V^2 dt
\end{array}
    $$
\end{defi}

We can therefore define the sub-Riemannian distance induced by $V$ on $\mathcal{Q}$ as the infimum of the length of horizontal curves :
\begin{defi}[Sub-Riemannian distance]
    Let $q_0,q_1\in\mathcal{Q}$. We define the sub-Riemannian distance $d_V(q_0,q_1)$ as
    $$
    d_V(q_0,q_1) = \inf_{\underset{q(0)=q_0,q(1)=q_1}{(q,u) \operatorname{horizontal}} } L(q,u)
    $$
\end{defi}


Furthermore, we have the following proposition that will help us find a minimum :
\begin{prop}
\label{dist_on_Q}
    The application $d_V$ is a true distance on $\mathcal{Q}$, and the topology induced by the sub-Riemannian distance is weaker than the intrinsic topology on $\mathcal{Q}$. Furthermore the distance $d_V$ is also equal to the infimum of the energy on horizontal curves, i.e. for $q_0, q_1 \in\mathcal{Q}$ :
    $$
    d_V(q_0,q_1) = \inf_{\underset{q(0)=q_0,q(1)=q_1}{(q,u) \operatorname{horizontal}} } \sqrt{2E(q,u)}
    $$
   where we  take the infimum over the set 
   $AC_{L^2}(I,\mathcal{Q})\times L^2(I,V)$. 
   
\end{prop}

\begin{proof}
    The proof is mostly contained in \cite{arguillere2020sub}, we include the proof here for sake of completeness. Let's prove first that the application $d_V$ separates points. Let $q_0, \ q_1 \in \mathcal{Q}$ distinct, a chart $U\subset \mathcal{Q}\rightarrow \mathbb{B}$ around $q_0$ with $(\mathbb{B}, |\cdot|_{\mathbb{B}})$ Banach space, and $\epsilon>0$ such that the open ball $B_{\mathbb{B}}(q_0,\epsilon)$ does not contain $q_1$. Since $\xi : U \rightarrow \mathcal{L}(V,TU) \simeq \mathcal{L}(V,\mathbb{B})$ is continuous, there exists an open ball $B' = B_{\mathbb{B}}(q_0,\epsilon')\subset B_{\mathbb{B}}(q_0,\epsilon)$, and $a>0$ such that for all $q\in B'$, we have $
    \lVert\xi_q \rVert_{\mathcal{L}(V,\mathbb{B})} \leq a
    $, i.e. for all $q\in B', u\in V$
    $$
    |\xi_q (u)|_{\mathbb{B}} \leq a |u|_V
    $$
    Now let $(q,u) : I \rightarrow \mathcal{Q}\times V$ horizontal system such that $q(0)=q_0$, and $q(1)=q_1$. As $q$ is continuous, there exists $t_0 \in I$ such that $|q(t_0)-q_0|_{\mathbb{B}}=\epsilon'$ and $q(t)\in B'$ for $t\leq t_0$. We get :
    \begin{equation*}
    \begin{split}
        L(q,u) &\geq L(q_{|[0,t_0]},u_{|[0,t_0]}) 
        \geq \int_0^{t_0}|u(t)|_V dt \\
        &\geq \frac{1}{a}\int_0^{t_0}|\xi_q (u)|_{\mathbb{B}} dt = \frac{1}{a}\int_0^{t_0}|\dot{q}(t)|_{\mathbb{B}} dt 
        \geq \frac{1}{a} |q(t_0)|_{\mathbb{B}} =\epsilon'/a
    \end{split}
    \end{equation*}
    Therefore we conclude that $d_V(q_0,q_1)>0$.

    To prove the last point, we first see by Cauchy-Schwarz inequality that for $(q,u)$ a horizontal system, we have :
    $$
    L(q,u)\leq\sqrt{2 E(q,u)}.
    $$
    with equality if and only if $(q,u)$ is constant speed (i.e. if $|u|_V$ is constant). Now for the reverse inequality, we approximate horizontal systems by constant speed reparametrization and show the length stay close. Indeed, let $(q,u)$ horizontal system with $u\in L^1(I,V)$. Then for $\epsilon>0$, we can consider the increasing absolutely continuous bijection $s_{\epsilon}(t)=\left( \int_0^t|u_t|_V + \epsilon t\right)/\left(|u|_{L^1} +\epsilon\right)$, with inverse $s\rightarrow t_{\epsilon}(s)$ such that $t'_{\epsilon}(s)=(|u|_{L^1}+\epsilon)/(|u_{t_{\epsilon}(s)}|_V+\epsilon$. Now we can define $\tilde{u}_s = t'_{\epsilon}(s)u_{t(s)}$ and see that $|\tilde{u}_s|_V\leq |u|_{L^1}+\epsilon$. Therefore $\tilde{u}$ is in $L^{\infty}(I,V)\subset L^2(I,V)$, and we also see that if $\tilde{q}_s = q_{t_\epsilon(s)}$, then :
    $$
    \dot{\tilde{q}}_s = \xi_{\tilde{q}_s} \tilde{u}_s
    $$
    so that $(\tilde{q},\tilde{u})$ is a $AC_{L^2}$ horizontal system with endpoints $q_0$ and $q_1$. We finally have that
    $$
    |\tilde{u}|_{L^2}\leq |u|_{L^1} + \epsilon
    $$
    which concludes the proof. 
  \end{proof}

We finish the section with some geodesic and metric completeness result. 
We first recall the definition of geodesics associated to  the distance $d_V$ :
\begin{defi}[Geodesics in $\mathcal{Q}$]
\label{geod_def}
Let $(q,u)\in AC_{L^2}(I,\mathcal{Q})\times L^2(I,V)$  an horizontal system. Then
\begin{itemize}
    \item  We say that the curve $(q,u)$ is a \emph{geodesic} if it minimizes locally the length, meaning for every $t_0\in I$, and $t_1$ close enough to $t_0$ :
    $$
    L((q,u)_{|[t_0,t_1]})=d_V(q(t_0),q(t_1))\,.
    $$
    \item The curve $(q,u)$ is a \emph{minimizing geodesic} if its total length is equal to the distance between the endpoints.
\end{itemize}
\end{defi}
\begin{rem}
    We already saw in proof of Proposition \ref{dist_on_Q} that if $(q,u)$ is a horizontal system that minimizes the energy, then it's immediately a minimizing geodesic, and is also parametrized with constant speed. Conversely, if $(q,u)$ is a minimizing geodesic parametrized with constant speed, then $(q,u)$ also minimizes the energy, and we have
$$
L(q,u)=\sqrt{2E(q,u)}\,.
$$
\end{rem}

We will need another asumption on the action of the groups $\{G^k,k\}$ to prove that $(\mathcal{Q},d_V)$ is a geodesic metric space, i.e. that we can join any two points of $\mathcal{Q}$ by a minimizing geodesic :
\begin{description}[leftmargin=*]
\item[(S.4)] For every $q\in\mathcal{Q}$, the endpoint mapping $\operatorname{end}_{q} : L^2(I,V)\rightarrow \mathcal{Q}, u\mapsto A_q\circ\operatorname{Evol}_{G^{k_0}}(u)(1)$ is weakly continuous where $\mathcal{Q}$ is equipped with some Hausdorff topology
\end{description}
\begin{rem}
    In most cases, the endpoint mapping $\operatorname{end} : L^2(I,V)\rightarrow G^{k_0}$ in $G^{k_0}$ is weakly continuous with regards to some Hausdorff topology in $G^{k_0}$, and thus we just need to study continuity of mapping $A_q$.
\end{rem}
\begin{theo}[Completeness]
    \label{completeness}
    \begin{enumerate}
        \item The space $\mathcal{Q}$ with distance $d_V$ is metrically complete.
        \item Moreover, if (S.4) is satisfied, then $\mathcal{Q}$ is a geodesic metric space, meaning for $q,q'\in \mathcal{Q}$ such that $d_V(q,q')<\infty$, there exists a minimizing geodesic connecting $q$ and $q'$.
    \end{enumerate}
     
\end{theo}

\begin{proof}
        We follow the proof from \cite{articleTro1995} in our more general context. We prove first $(\mathcal{Q},d_V)$ is complete metric space. We consider a Cauchy sequence $(q_n)_{n}$ in $\mathcal{Q}$, and we can suppose $\sum_n d_V(q_n,q_{n+1}) < \infty$. Thus there exists $u_n \in L^1(I,V)$, such that $q_{n+1}=\operatorname{Evol}_{G^{k_0}}(u_n)\cdot q_n$ and such that $|u_n|_{L^1(I,V)}<2d_V(q_n,q_{n+1})$. We define a new sequence $w_n$ in $L^1(I,V)$ such that 
        $$
        w_n = \left\{\begin{array}{l}
             2^{k+1} u_k(2^{k+1} (t-1)+2) \ \mbox{ for} \ t \in [\frac{2^k-1}{2^k}, \frac{2^{k+1}-1}{2^{k+1}}], k\in \{0,1,..,n \} \\
             0 \ \mbox{for} \ t \in [\frac{2^{n+1}-1}{2^{n+1}}, 1] 
        \end{array}\right.
        $$
        Intuitively, we just concatenate the paths $u_n$, so that $q_{n+1}= \operatorname{end}_{q}(w_n)$, and $|w_n|_{L^1(I,V)}=\sum_{k\leq n}|u_k|_{L^1(I,V)}<\infty$ and $|w_{n+1}-w_n| = |u_{n+1}|_{L^1(I,V)}$. Therefore $(w_n)$ is a Cauchy sequence in the Banach space $L^1(I,V)$, and thus converges to $w_{\infty}\in L^1(I,V)$. We denote $q_{\infty}=\operatorname{end}_{q_0}(w_n)\in\mathcal{Q}$. We also see that $q_{\infty}=\operatorname{end}_{q_n}(w_{\infty}-w_{n-1})$, and thus 
        $$
        d_V(q_{\infty},q_n)\leq |w_{\infty}-w_{n-1}| \rightarrow 0
        $$
        Therefore $(\mathcal{Q},d_V)$ is a complete metric space.

        Now we prove the existence of geodesic between points. Let $q,q'\in \mathcal{Q}$ such that $d_V(q,q')<\infty$, and let $(q_n,u_n)$ a minimizing sequence for the energy, with $u_n\in L^2(I,V)$. Since $(u_n)$ is bounded in $L^2$, we can suppose, up to a subsequence, that $(u_n)$ converges weakly towards $u_{\infty}\in L^2(I,V)$. We denote $q_{\infty}$ the horizontal curve such that $\dot{q}_{\infty} = \xi_{q_\infty}u_\infty$. Since the endpoint mapping $\operatorname{end}_q$ is weakly continuous, and $q_n(1)=q'$ for all $n$, then we also have $q_\infty = q'$. Finally, since $|u_\infty|_{L^2}\leq\lim\inf |u_n|_{L^2} = d_V(q,q')$, we get the result.
\end{proof}

\subsection{Sub-Riemannian geodesics and critical points of the energy}

The aim of this section is to characterise geodesics and critical points of the energy in a sub-Riemannian setting, as in \cite{arguillere_trelat_2017}.

We fix $q_0,q_1 \in \mathcal{Q}$. We first define and study the space of horizontal systems connecting $q_0$ and $q_1$ We consider the endpoint mapping 
$$
\operatorname{end}_{q_0} : \left|
\begin{array}{ccl}
   L^2(I,V) \ (\simeq  \operatorname{Hor}_{q_0}(I) )  & \rightarrow & \mathcal{Q}  \\
     u & \mapsto & q^u(1)
\end{array}
\right.
$$
The endpoint mapping is $C^1$ as the evaluation $q\in AC_{L^2}(I,\mathcal{Q})\mapsto q(1)\in \mathcal{Q}$ is smooth (Prop \ref{propA:ev}). The space of horizontal systems with endpoints $q_0$ and $q_1$ is defined as : $$\mbox{Hor}_{q_0,q_1}(I)=\operatorname{end}_{q_0}^{-1}(\{q_1\})$$ This space is not in general a submanifold of $L^2(I,V)$ since the mapping $\operatorname{end}_{q_0}$ is not necessarily a submersion.

We can now define a notion of sub-Riemannian geodesic to characterize the minimizers of the energy. This was described in particular in \cite{arguillere_trelat_2017,arguillere2020sub,https://doi.org/10.48550/arxiv.1504.01767}. Suppose $(q,u)$ is a miminum of the energy with endpoints $q_0$ and $q_1$ (i.e. the curve $(q,u)$ is a minimizing geodesic). Therefore the mapping $u\mapsto (E(u),\operatorname{end}_{q_0}(u))$ is not surjective from an open neighborhood of $u$ onto an open neighborhood of $(E(u),q_1)$. Therefore the differential $(dE(u),d\operatorname{end}_{q_0}(u))$ is also not surjective. This can lead to three different cases whether the mapping $(dE(u),d\operatorname{end}_{q_0}(u))$ has closed range stricly included in $\mathbb{R}\times T\mathcal{Q}$ or whether it is dense. We will focus only on the case of normal geodesics when we can define Lagrange multipliers, as they are in particular the critical points of the inexact matching problem.
\begin{defi}[Sub-Riemannian normal geodesic]
Let $(q,u)\in\mbox{Hor}_{q_0}(I)$ an horizontal system.

We say that $(q,u)$ is a \emph{sub-Riemannian normal geodesic} (or just normal geodesic) if there exists Lagrange multipliers $(\lambda_0,\lambda)\in \mathbb{R}\times T_{q(1)}\mathcal{Q}^*$, with $\lambda_0$ not equal to zero such that
\begin{equation}
    \lambda_0 dE(u)\delta u + \left(\lambda\,|\,d\operatorname{end}_{q_0}(u)\delta u\right)=0
\end{equation}
\end{defi}
\begin{rem} Normal geodesics are in particular geodesics in the sense of definition \ref{geod_def}, i.e. they minimize locally the energy \cite[Theorem 7]{arguillere2020sub}. They correspond to riemannian geodesics in classic riemannian geometry. In the case where there exists non zero Lagrange multipliers with $\lambda_0=0$, we call the horizontal system $(q,u)$ a singular curve (cf. \cite{arguillere2020sub}). Singular curves can also be geodesics, and even normal geodesics \cite[section 5]{montgomery2002tour}
\end{rem}


\subsection{Hamiltonian flow and normal geodesics}
In this part, we want to study the normal geodesics and determine a geodesic equation through the Hamiltonian formulation.

We introduce the Hamiltonian of the system :
$$
\mathcal{H}_{\mathcal{Q}} : 
\left|
  \begin{array}{ccl}
    T\mathcal{Q}^* \times V & \longrightarrow & \R \\
    (q, \ p, \ u) &\longmapsto & \left(p\,|\,\xi_q(u)\right) - \frac{1}{2}|u|_V^2 \\
  \end{array}
\right.
$$
By hypothesis (S.3), the mapping $\xi :  T_eG^{k_0+2}\times Q\to TQ$ given by $(q,u) \mapsto \xi_q (u) = u \cdot q$ is $C^2$.
Moreover, the application $((q,p),(q,X)) \in T\mathcal{Q}^* \oplus_Q T\mathcal{Q} \mapsto \left(p\,|\,X\right) \in \mathbb{R}$ is smooth, and therefore the Hamiltonian $\mathcal{H}_{\mathcal{Q}} : T\mathcal{Q}^* \times V \longrightarrow \R$ is $C^2$. In local coordinates, the partial derivative is given by $\partial_p \mathcal{H}_\mathcal{Q}(q,p,u)$ :
$$
\forall \delta p \in T_q\mathcal{Q}^*, \ \partial_p \mathcal{H}_{\mathcal{Q}}(q,p,u) \delta p= (\delta p\,|\,u\cdot q)
$$
so that $\partial_p \mathcal{H}_{\mathcal{Q}}(q,p,u)\simeq u\cdot q = \xi_q (u) \in T_q\mathcal{Q}$. Therefore there exists a (partial) symplectic gradient $\nabla^{\omega}\mathcal{H}_{\mathcal{Q}}(q,p,u)$ for every $u\in V$ (that we define introducing the Liouville form and its exterior derivative \cite{ARGUILLERE2015139}). In canonical charts of $T\mathcal{Q}^*$, we have :
$$
    \nabla^{\omega}\mathcal{H}_{\mathcal{Q}}(q,p,u) = (\partial_p \mathcal{H}_{\mathcal{Q}}(q,p,u), \ -\partial_q \mathcal{H}_{\mathcal{Q}}(q,p,u))
$$
We get the following result :
\begin{theo}[Hamiltonian flow] Let $q_0,q_1 \in \mathcal{Q}$.

Then an horizontal system $(q,u)\in\operatorname{Hor}_{q_0,q_1}(I)$ is a normal geodesic if and only if there exists $t\mapsto p(t)\in T_{q(t)}\mathcal{Q}^*$ in $AC_{L^1}(I,T\mathcal{Q}^*)$ such that $(q,p,u)$ satisfies the Hamiltonian equations :
\begin{equation}
    \label{Ham_eq}
    \left\{ \begin{array}{ll}
       \left(\dot{q}, \ \dot{p}\right)   =  \nabla^{\omega}\mathcal{H}_{\mathcal{Q}}(q,p,u) \\
        \partial_u\mathcal{H}_{\mathcal{Q}}(q,p,u)  =  0 
    \end{array}
    \right.
\end{equation}
\end{theo}
\begin{proof}
Let $p_1\in T_{q_1}\mathcal{Q}^*$, and we introduce the mapping :
$$F : u\in L^2(I,V) \mapsto E(u) - \left(p_1\,|\,\operatorname{end}_{q_0}(u)\right)$$

Let $p\in AC_{L^1}(I,T\mathcal{Q}^*)$ be the solution for $u\in L^2(I,V)$ of the following linear Cauchy problem :
\begin{equation}
\label{CLdual}
    \left\{
    \begin{array}{ll}
    \dot{p}(t)=-\partial_q \mathcal{H}_{\mathcal{Q}}(q(t),p(t),u(t))=-\big(\partial_q(\xi_{q} u(t))_{q=q(t)}\big)^*p(t) \\
    p(1)=p_1
    \end{array}
    \right.
\end{equation}
    We now compute the differential of the mapping 
    $F$
    and we prove that for all $\delta u \in L^2(I,V)$,
    $$
    dF(u)\delta u = -\int_I \partial_u\mathcal{H}(q(t),p(t),u(t))\delta u dt
    $$
    Let $\delta u \in L^2(I,V)$, we have 
    \begin{align*}
        dF(u)\delta u &= dE(u)\delta u - \left(p_1\,|\,d\operatorname{end}_{q_0}(u)\delta u\right) \\ &= \int_I\langle u , \delta u \rangle_V dt - \left(p(1)\,|\,d\operatorname{end}_{q_0}(u)\delta u\right)
    \end{align*}
We recall that $\delta q = \partial_u q^u \delta u$ satisfies the linear Cauchy problem \eqref{linq2}

$$
\delta q(0) = 0, \ \delta \dot{q}(t) - \partial_q \left(\xi_{q}u(t)\right)_{|q=q^u(t)}\delta q(t) = \xi_{q^u(t)}(\delta u(t))
$$
with
$$
d\operatorname{end}_{q_0}(u)\delta u = \delta q(1)\,.
$$
Now, as $t\mapsto p(t)$ is also solution of the linear Cauchy equation \eqref{CLdual} and
    using integration by part, we find that
    \begin{align*}
    (p(1)&\,|\,d\operatorname{end}_{q_0}(u)\delta u) = \left(p(1)\,|\,\delta q(1)\right) \\
    &= \left(p(0)\,|\,\delta q(0)\right) + \int_I \left(\dot{p}(t)\,|\,\delta q(t)\right) + \left(p(t)\,|\,\delta\dot{q}(t)\right)dt\\
    &=\int_I-\bigg(\left(\partial_q(\xi_{q} u(t))_{|q=q^u(t)}\right)^*p(t)\,|\,\delta q(t)\bigg)+\bigg(p(t)\,|\,\partial_q \left(\xi_{q}u(t)\right)_{q=q^u(t)}\delta q(t) + \xi_{q^u(t)}(\delta u(t)) \bigg) dt \\
    &=\int_I \left(p(t)\,|\,\xi_{q^u(t)}(\delta u(t))\right) dt 
    \end{align*}
Finally this gives us
\begin{align*}
    dF(u)\delta u &= \int_I\langle u , \delta u \rangle_V dt - \int_I \left(p(t)\,|\,\xi_{q^u(t)}(\delta u(t))\right) dt  \\
    &= -\int_I \partial_u\mathcal{H}(q(t),p(t),u(t))\delta u dt
\end{align*}
Therefore we have the following equivalence :
\begin{equation}
    dE(u) = d\operatorname{end}_{q_0}(u)^*p_1 \iff \forall t,\  \partial_u \mathcal{H}(q(t),p(t),u(t))=0
\end{equation}
which concludes the proof by definition of the normal geodesics.
\end{proof}

\subsection{Inexact matching}

We might want to consider an inexact matching problem by minimizing on the set $\operatorname{Hor}_{q_0}(I)$ :
\begin{equation}
    J(u) = E(u) + g(q^u(1)) = E(u) + g(\operatorname{end}_{q_0}(u))
\end{equation}
where $g:\mathcal{Q}\rightarrow \R$ is a mapping that measures the distance with $q_1$. In this case we minimize on the whole vector space $\operatorname{Hor}_{q_0}(I) \simeq L^2(I,V)$ and we first start by proving the minimum of $J$ is attained :
\begin{prop}[Existence of minimizers for the inexact matching problem]
    Suppose (S.4) is satisfied and $g$ is continuous for the same Hausdorff topology. Then there exists $(q,u)\in \operatorname{Hor}_{q_0}(I)$ such that $J(u)$ is minimal.
\end{prop}
\begin{proof}
    The proof follows the proof of theorem \ref{completeness}. We introduce a minimizing sequence $(q_n,u_n)\in\operatorname{Hor}_{q_0}(I)$. The sequence $(u_n)$ converges weakly to $u_\infty\in L^2(I,v)$, and we denote $q_\infty\in AC_{L^2}(I,\mathcal{Q})$ such that $\dot q_{\infty} = \xi_{q_\infty}u_\infty$. Since $\operatorname{end}_{q_0}$ is weakly continuous, and $g$ continuous, we get that $g(q_\infty(1)) = \lim g(q_n(1))$. Moreover, the lower semi-continuity of the $L^2$ norm gives 
    \begin{align*}
        J(u_\infty) = |u_\infty|_{L^2} + g(q_\infty(1)) \leq \lim\inf |u_n|_{L^2} + \lim g(q_n(1)) \leq \lim \left( |u_n|_{L^2} + g(q_n(1)) \right)
    \end{align*}
    Hence the result.
\end{proof}

We can also obtain a characterization of the critical points of $J$ on $\operatorname{Hor}_{q_0}(I)$
\begin{theo}[Critical points of $J$] Assume that $g$ is  $C^1$.
    Let $(q,u)\in \operatorname{Hor}_{q_0}(I)$. Then if $(q,u)$ is a critical point of $J$ it is a normal geodesic.
\end{theo}
\begin{proof}
    Let $(q,u)\in \operatorname{Hor}_{q_0}(I)$ a critical point of $J$, i.e, we have :
    $$
    dJ(u) = dE(u) + dg(q(1))d\operatorname{end}_{q_0}(u) = 0
    $$
    Therefore $(q,u)$ is a normal geodesic with momentum $p(1)=dg\left(q(1)\right)$
\end{proof}

In the following, we will therefore only study the case of normal sub-Riemannian geodesics.

\section{Intermezzo: Multiscale shape space}
As a needed break along the development of the theory and an important illustrative example of the use of GGA framework in a specific setting, we address in this section the case of multiscale shape spaces for registration through the action of the product of diffeomorphisms.

\subsection{Basic framework}
  Let $\boldsymbol{\mathcal{Q}}=\prod_{1\leq l \leq L}\mathcal{Q}_l$ Banach manifold, with $L\geq 0$. We also introduce the group of diffeomorphisms
  $$
  \operatorname{Diff}_{C_0^k}(\R^d)=\left(\id + C_0^k(\R^d,\R^d)\right)\cap\operatorname{Diff}^1(\R^d)
  $$
  where $C_0^k(\R^d,\R^d)$ is the space of $C^k$ mapping whose derivatives up to order $k$ are vanishing at infinity.
  \begin{prop}
    The family of groups $\operatorname{Diff}_{C_0^k}(\R^d)$ satisfies the conditions (G.1-5) p.\pageref{HypGroup}.
\end{prop}
\begin{proof}
The group $\operatorname{Diff}_{C^k_0}$ is an open subset of the affine Banach space $\id+C^k_0(\R^d,\R^d)$, and therefore $\bG^k$ is a Banach manifold. Each group $\operatorname{Diff}_{C_0^k}$ satisfies conditions (G.1-5) (cf. \cite{bauer2023regularity}, \cite{kriegl1997convenient} for example), and therefore the group $\operatorname{Diff}_{C^k_0}$ satisfies those conditions too.
\end{proof}
\begin{rem}The group $\operatorname{Diff}_{C_0^k}(\R^d)$ is in particular a half-Lie group.
\end{rem}
  We suppose that the group $\operatorname{Diff}_{C_0^k}(\R^d)$ acts on each layer $\mathcal{Q}_l$ verifying (S.1-3) conditions. We now want to study sub-Riemannian geometries on $\boldsymbol{\mathcal{Q}}$ induced by right-invariant metrics on diffeomorphisms groups to define a multi-scale version of Large Diffeormorphic Metric Mapping (LDDMM). Such approaches were first described in \cite{doi:10.1137/110846324,risser2011simultaneous,sommer2013sparse}. In those papers, the authors introduce a family of kernels $(K^l)_{1\leq l \leq L}$ and define the space of controls as a reproducing kernel Hilbert space (RKHS) for the associated kernel $\sum_l K^l$. The matching problem is shown to be equivalent to
\begin{equation}
    \inf_{v^1,\ldots,v^L} \frac{1}{2}\sum_{l=1}^L\int_I |v^l|^2_{V_l} dt + g(\varphi_1\cdot q^L_0)
    \label{Multiscale1}
\end{equation}
with dynamic
$$
\left\{\begin{array}{l}
     \varphi_0 = \id  \\
     \dot{\varphi}_t = (\sum_{l=1}^L v^l)\circ \varphi_t
\end{array}\right.
$$ where 
we suppose that we have a sequence of continuous embeddings $V_1 \hookrightarrow \cdots V_{L-1}\hookrightarrow  V_L\hookrightarrow  C_0^m(\R^d)$ associated to each kernel $K^l$.
We can therefore introduce the product space $\bV = \prod_l V_l$ equipped with the Hilbert norm defined by $|\bv|^2_{\bV}=\sum_{l=1}^{L}|v^l|^2_{V_{l}}$ for $\bv=(v^l)_{1\leq l <L}\in\bV$. To deal with problem \eqref{Multiscale1}, we introduce another Hilbert norm on $\bV$ given by $|\bu|_A = |A\bu|_{\bV}$ where $A : \bV \rightarrow \bV$ linear isomorphism such that $A\bu = \bv =  (u^1, u^2-u^1,\ldots ,u^{L}-u^{L-1})$. We can now introduce family of groups 
$$
\bG^k = \prod_{1\leq l\leq L} \mbox{Diff}_{C_0^k}(\R^d)
$$ and the matching problem \eqref{Multiscale1} is equivalent to study sub-Riemannian geodesics induced by the norm $|\cdot|_A$ on $\bG^k$, but with endpoint constraints only on the finest scale. We define the Hamiltonian on $\mathcal{Q}_L$ for problem $\eqref{Multiscale1}$ 
\begin{equation}
    \mathcal{H}_{\mathcal{Q}_L} \left(q, p, \bu\right) = \left(p\,|\,\xi_q(u^L)\right) - \frac{1}{2} |A\bu|_{\bV}^2
\end{equation}

We get the following 
\begin{prop}
    \label{Multiscale2}
    Let $t\mapsto (q(t),p(t), \bu(t))$ be a normal geodesic associated to the Hamiltonian $\mathcal{H}_{\mathcal{Q}_L}$, and let $\bv= A\bu$. Then we get 
    $$
    v^l = K^l\left(\xi_{q}^*p \right)
    $$
    and
    $$
    u^L = \sum_{l\leq L}K^l\left(\xi_{q}^*p\right)
    $$
    \label{control_multi1}
\end{prop}
\begin{proof}
The proof is mainly included in \cite{doi:10.1137/110846324}, but we recall it for sake of completeness. We compute $\partial_{\bu}\mathcal{H}_{\mathcal{Q}_L} \left(q, p, \bu \right)$. 
We denote $L^l=(K^l)^{-1}$ for $1\leq l\leq L$, 
and $L_{\bV}:\bV\to \bV'$ the Riesz canonical isometry. Let $\delta \bu \in \bV$ we get :
\begin{align*}
    \partial_\bu\mathcal{H}_{\mathcal{Q}_L} \left(q, p, \bu \right) \delta \bu &= \left(p\,|\,\xi_q(\delta u^L)\right) - \left( L_\bV A\bu\,|\,A\delta \bu\right) \\
    &= \left(p\,|\,\xi_q(\delta u^L)\right) - \sum_{l=2}^{L}(L^l v^l\,|\,\delta u^l - \delta u^{l-1}) - (L^1 v^1\,|\,\delta u^1)\\
    &= \left(p\,|\,\xi_q(\delta u^L)\right) - \sum_{l=1}^{L-1}(L^l v^l - L^{l+1}v^{l+1}\,|\,\delta u^l) - (L^L v^L\,|\,\delta u^L) 
\end{align*}
In particular, the optimality condition $\partial_u\mathcal{H}_{\mathcal{Q}_L} (\delta \bu) = 0$ implies 
$$
\left\{\begin{array}{l}
      L^lv^l - L^{l+1}v^{l+1} = 0, \ 1\leq l\leq L-1 \\
      \xi_q^*p = L^Lv^L  
\end{array}\right.
$$ and therefore $v^l = K^l \xi_q^*p$ for all $1\leq l \leq L$. As $u^L=\sum_{l=1}^L v^l$, we finally get the result.
\end{proof}

More recently in \cite{BME}, the authors study the case where the dynamic for all scale is controlled, i.e. the term $g(\varphi_1\cdot q_0^L)$ is replaced by $g(\boldsymbol{\varphi_1}\cdot\bm{q_0})$ with $\bm{q_0}=(q_l)_l\in \boldsymbol{\mathcal{Q}}$ :

\begin{equation}
    \inf_{v^1,\ldots v^L} \frac{1}{2}\sum_{l=1}^L\int_I |v^l|^2_{V_l} dt + g(\boldsymbol{\varphi_1}\cdot\bm{q_0})
    \label{Multiscale2_eq}
\end{equation}
with dynamic.
$$
\left\{\begin{array}{ll}
     \varphi_0^1=\ldots =\varphi_0^L = \id  \\
     \dot{\varphi}^l_t = (\sum_{k=1}^l v^k)\circ \varphi^l_t & \forall l \leq L
\end{array}\right.
$$ 

This can be reformulated more simply as 
\begin{equation}
  \label{Multiscale2b_eq}
    \inf_{\bu \in \bV} \int_I|A\bu|^2_\bV + g(\boldsymbol{\varphi_1}\cdot\bm{q_0})
\end{equation}
where $\dot{\boldsymbol{\varphi}}=\bu\circ\boldsymbol{\varphi}$. We get for $\bp=(p^l)_{1\leq l\leq L}$ and $\bq=(q^l)_{1\leq l\leq L}$ the corresponding Hamiltonian

\begin{equation}
    \mathcal{H}_{\boldsymbol{\mathcal{Q}}} \left(\bq, \bp, \bu \right) = \sum_{l=1}^L\left(p^l\,|\,\xi_{q^l}(u^l)\right) - \frac{1}{2} |A\bu|_{\bV}^2\,.
\end{equation}

We get the following.
\begin{prop}
    Let $t\mapsto \left(\bq(t),\bp(t),\bu(t)\right)$ a normal geodesic associated to the Hamiltonian $\mathcal{H}_{\boldsymbol{\mathcal{Q}}}$, and let $\bv = A\bu$. Then we get 
$$
\begin{array}{ccc}
v^l = K^l\left(\sum_{m\geq l} \xi_{q^m}^*p^m \right)
     & \mbox{ and } &
     u^l = \sum_{k\leq l}K^k\left(\sum_{m\geq k} \xi_{q^m}^*p^m\right)
\end{array}
$$
    \label{control_multi2}
\end{prop}
\begin{proof}
    The proof is similar to the proof of proposition \ref{control_multi1} and follows from equality $\partial_\bu\mathcal{H}_{\boldsymbol{\mathcal{Q}}} = 0$, where
$$
    \partial_\bu\mathcal{H}_{\boldsymbol{\mathcal{Q}}} \left(\bq,\bp,\bu\right)\delta\bu = \sum_{l=1}^L\left(p^l\,|\,\xi_{q^l}(\delta u^l)\right) - \sum_{l=2}^{L}(L^l v^l\,|\,\delta u^l - \delta u^{l-1}) - (L^1 v^1\,|\,\delta u^1)
$$
\end{proof}
\begin{rem}
  Originating from \cite{BME}, the aforementioned setting could be also compared to another multiscale approach coupling hierarchical multiscale image and diffeomorphism decomposition in a registration setting \cite{tadmor2004multiscale,MODIN20191009,debroux2023multiscale}. In the later framework, the approach relies on a greedy sequential decomposition of residuals from coarse to fine scales. The resulting output is a deformation achieved by the composition of a sequence of diffeomorphisms. This approach effectively integrates image and diffeomorphism decomposition within a well-defined setting specifically designed from image registration. However, it does not align with the sub-riemannian approach since the scales are considered sequentially in time ratehr than in parallel. The simultaneous action of all the scales through time is a distinctive feature of the setting \eqref{Multiscale2_eq} and \eqref{Multiscale2b_eq} that allows its integraion into the (GGA) framework and enables the production of time homogeneous trajectories at different resolutions.
\end{rem}
\def\Glie{\bm{G}_0}
\def\fG{\mathfrak{g}_0}
\def\bg{\bm{a}}
\def\bs{\bm{s}}
\def\bpg{\bm{p}_{\bm{a}}}
\def\bK{\bm{K}}
\def\bG{\bm{G}}
\def\bqS{\bm{q_S}}
\def\bqT{\bm{q_T}}
\def\bQ{\boldsymbol{\mathcal{Q}}}
\def\bvphi{\boldsymbol{\varphi}}
\def\Sod{SO(d)}
\subsection{Adding the action of a finite dimensional Lie group}

In classical LDDMM, a first step is often performed to align source and target objects through rigid motions. In the multiscale setting, one can actually add rigid motion alignment as an additional layer, a priori the coarsest layer. However, even for the simple case of rigid motion, the underlying Lie algebra is not embedded in any $C_0^k(\R^d,\R^b)$ Banach spaces since the vanishing conditions at infinity are not fulfilled. More interestingly, the pre-alignment of the target by the action of low dimensional group of symmetries comes essentially to alleviate a nuissance low dimensional group of transformations from the analysis. As a consequence, it is more natural to consider the group product
$$K^k=\bG^k\times \Glie$$
where  $\Glie$ is a finite dimensional Lie group of transformations $\bg\in \Glie$ with the component-wise action on pairs $(\bqS,\bqT)\in\bQ\times\bQ$ given by $(\bvphi,a)\mapsto (\bvphi\cdot \bqS,\bg\cdot \bqT)$  where $\bQ=\prod_{1\leq l \leq L}\mathcal{Q}_l$ . This fits in our general framework~:

\begin{prop}
  \begin{enumerate}
  \item The family of groups $\{\bm{K}^k, k\in \mathbb{N}\}$ satifies
    conditions (G.1-5) p.\pageref{HypGroup}.
  \item If the action of $\Glie$ on $\bQ$ satisfies (S.1-3), then the action of $(\bK^k)_{k\geq 1}$ on $\bQ\times \bQ$ satisties (S.1-3). 
  \end{enumerate}
\end{prop}
\begin{proof}
  (1) The result is straightforward, since $\operatorname{Diff}_{C_0^k}$ satisfies  conditions (G.1-5), and $\Glie$ is finite dimensional Lie groups so immediately satisfies conditions (G.1-5) (here the family would just consists in one group $\{\Glie\}$). The result follows immediately as $\bm{K}^k$ has component-wise law of composition.\\
 (2) follows immediatly since the $\bK^k$ has a component-wise action on $\bQ\times\bQ$.
\end{proof}

We can, considering the Lie algebra $\fG$ of $\Glie$ equipped with a dot product $\langle\,,\,.\rangle_{\fG}$ now define the matching problem including the action of finite dimensional symmetries $\bg\in\Glie$~:
\begin{equation}
    \inf_{(\bu,\bs)\in L^2(I,\bV\times \fG)} \frac{1}{2}\int_I\left( |A\bu(t)|_{\bV}^2 +\left|\bs(t)\right|_{\fG}^2\right)dt + g\left(\left(\boldsymbol{\varphi}(1)\cdot\bm{q_S},\bg(1)\cdot\bm{q_T}\right)\right)
    \label{rigid1}
\end{equation}
with dynamic given by 
$$
\left\{\begin{array}{ll}
     \dot{\bg}(t) = \bs(t)\cdot \bg(t)  \\
     \dot{\varphi}^{\ell}(t) = (\sum_{k=1}^{\ell} v^k)\circ \varphi^{\ell}(t) & \forall \ell \leq L
\end{array}\right.
$$
with $\bs(t)\in\fG$. We can therefore introduce the Hamiltonian of the optimal control problem 
\begin{align*}
    \mathcal{H}_{\boldsymbol{\mathcal{Q}}\times\boldsymbol{\mathcal{Q}}}\left(\left(\bm{q_S},\bm{q_T}\right),\left(\bm{p_S},\bm{p_T}\right),\left(\bu,\bs\right)\right) &= \left(\bm{p_S}|\bu\cdot \bm{q_S}\right) + \left(\bm{p_T}|\bs\cdot \bm{q_T}\right) - \frac{1}{2}\left( |A\bu|_{\bV}^2+ \left|\bs\right|_{\fG}^2\right) 
\end{align*}
with $\bm{q_S}=(q_S^{\ell})_{\ell}, \bm{q_T}=(q_T^{\ell})_{\ell}\in\boldsymbol{\mathcal{Q}}$, $\bm{p_S}=(p_S^{\ell})_{\ell}, \bm{p_T}=(p_T^{\ell})_{\ell}\in\boldsymbol{\mathcal{Q}}^*$ and $\bu=(u^{\ell})_{1\leq \ell \leq L}\in \prod_{1\leq \ell \leq L}T_eG^k$. Equivalently, we can assimilate $\bm{q_T}$ with finite dimensional motion $\bg$ (through $\bm{q_T}(t)=\bg(t)\cdot \bm{q_T}(0)$) and get the equivalent Hamiltonian :
\begin{equation*}
 \mathcal{H}_{\bQ\times\Glie}(\bm{q_S},\bg),\left(\bm{p_S},\bpg\right),\left(\bu,\bs\right))=\mathcal{H}_{\boldsymbol{\mathcal{Q}}}\big(\bm{q_S},\bm{p_S},\bu\big)+ \mathcal{H}_{\Glie}(\bg, \bpg, \bs)
 \end{equation*}
 with
 \begin{equation}
   \label{eq:10.12.1}
  \mathcal{H}_{\boldsymbol{\mathcal{Q}}}\left(\bm{q_S},\bm{p_S},\bu\right) = \sum_{l=1}^{L}\left(p_S^{\ell}|\xi_{q^{\ell}}(u^{\ell})\right)  - \frac{1}{2}|A\bu|_{\bV}^2  \,\text{ and }\, \mathcal{H}_{\Glie}(\bg, \bpg, \bs)=(\bpg\,|\,\bs\cdot\bg)-\frac{1}{2} \left|\bs\right|_{\fG}^2\,.
 \end{equation}
\def\SOd{SO(d)}
\def\Simd{\operatorname{Sim}^+(\R^d)}
\def\fsimd{\mathfrak{sim}(\R^d)}
\def\fsod{\mathfrak{so}(d)}
\subsubsection{The example of landmarks}
We apply the previous approach on the space of landmarks. We define the multiscale space as $\bQ=\prod_{1\leq \ell \leq L}(\R^d)^{I_{\ell}}$ with $(I_{\ell})_{\ell}$ an increasing sequence of index sets, and the following classical action :
$$
\varphi^{\ell}\cdot (q_i^{\ell}) = \left(\varphi^{\ell}(q_i^{\ell})\right)  \ \mbox{for } \ell\geq 1
$$
We consider the case $\Glie=\Simd=\R^*\times \SOd\times\R^d$ the group of orientation-preserving scaling, translation and rotations with the product $\bg\bg'=(\rho\rho',RR',\tau+\tau')$ for $\bg=(\rho,R,\tau)$ and $\bg'=(\rho',R',\tau')$ and the group action of on $(\R^d)^{I_{\ell}}$ defined by~:
$$
(\rho,R,\tau)\cdot q_i^{\ell} = \rho R(q_i^{\ell}-q_c^{\ell})+q_c^{\ell}+\tau  \ \mbox{for } i\in I_{\ell}
$$
with $q_c^{\ell}=\frac{1}{|I_{\ell}|}\sum_{i\in I_{\ell}}q_i^{\ell}$ the center of mass. We suppose $\Simd$ is equipped with the standard right-invariant metric associated with the metric on $\fG=\fsimd=\R\times\fsod\times\R^d$~:
$$
\langle (\alpha,r,\sigma),(\alpha',r',\sigma')\rangle_{\fsimd} = \alpha\alpha' +\langle r,r'\rangle +\langle\sigma, \sigma'\rangle
$$
where $\alpha,\alpha'\in \R$, $\sigma,\sigma'\in \R^d$, and $r,r'\in \fsod$ skew-symmetric matrices and $\langle r,r'\rangle=\text{Tr}(r^Tr')$.
We take a $L^2$ endpoint constraints $g(\bm{q},\bm{q'})= \frac{1}{2} \sum_{l=0}^{\ell}\sum_{i\in I_{\ell}}|q_i^{\ell}-q_i'^{\ell}|^2$. The control problem becomes 

\begin{equation}
  \begin{split} \mathcal{H}_{\bQ}\left(\bm{q},\bm{p},\bu\right) = \sum_{l=1}^{L}\sum_{i\in I_{\ell}}&\left(p_{i}^{\ell}|u^{\ell}(q_{i}^{\ell})\right) -\frac{1}{2}|A\bu|_{\bV}^2\,.\\
  \mathcal{H}_{\Simd}(\bg,\bpg,\bs)& = (p_\rho|\alpha\rho)+ \left(p_R|rR\right)- \frac{1}{2} |\bs|_{\fsimd}^2+\left(p_{\tau}|\sigma\right) 
\end{split}
\end{equation}
We recall the solutions of the optimal control problem satisfies the hamiltonian equations
$$
\left\{\begin{array}{l}
     \dot{\bm{q}} =  \partial_{\bm{p}}\mathcal{H}_{\bQ}(\bm{q},\bm{p},\bu),\ \dot{\bg}=\partial_{\bpg}\mathcal{H}_{\Simd}(\bg,\bpg,\bs)\\
         \dot{\bm{p}}=  -\partial_{\bm{q}}\mathcal{H}_{\bQ}(\bm{q},\bm{p},\bu),\ \dot{\bpg}=-\partial_{\bg}\mathcal{H}_{\Simd}(\bg,\bpg,\bs)\\
      \partial_{\bu} \mathcal{H}_{\bm{Q}}(\bm{q},\bu)=0,\ \partial_{\bs}\mathcal{H}_{\Simd}(\bg,\bpg,\bs)=0
\end{array}\right.
$$
\begin{prop}
    Let $\bm{q}_S,\bm{q}_T \in \boldsymbol{\mathcal{Q}}$ be the source and target objects and $\bm{e}=(1,0,\text{I}_d)$ be the neutral element of $\Glie$. Then the solution $t\mapsto \left( (\bm{q}(t),\bg(t)),(\bm{p}(t),\bpg(t))\right)$ of the optimal problem starting from $(\bm{q}_S,\bm{e})$ with endpoint constraints satisfies 
    $$
\left\{\begin{array}{l}
     \dot{q}_{i}^{\ell} = u_t^{\ell}(q_{i}), \ \ell \geq 1 \\
     \dot{p}_{i}^{\ell} = -\left(du_t^{\ell}(q_{i}^{\ell})\right)^*p_{i}^{\ell} \\
    \dot{\rho}=\alpha \rho, \dot{R} = r R, \dot{\tau} = \sigma \\
     (\dot{p}_{\rho}, \dot{p}_{R},\dot{p}_{\tau}) = (-\alpha p_{\rho},-r^Tp_{R},0) \\     
\end{array}\right.
$$
with $\alpha \in \R$, $r\in\fsod$, $\sigma\in\R^d$ and $\bu\in\bV$ satisfying 
$$
    u^{\ell} = \sum_{k\leq \ell}\bm{K}^k\left(\sum_{m\geq k}\sum_{i\in I_m} \delta_{q_{i}^m}^{p_i^m}\right)
    $$
    where $\delta_x^y\in C_0(\R^d,\R^d)'$ is defined by $(\delta_x^y\,|\, u)=\langle u(x),y\rangle$ for $x,y\in\R^d$.
Moreover, we have the following endpoint conditions for the coadjoint map
\begin{equation}
\label{endpoint_eq}
\left\{\begin{array}{lcl}
         p_{i}^{\ell}(1) &= &- q_{i}^{\ell}(1) + (\rho R)(1) (q_{T,i}^{\ell}-q_{T,c}^{\ell})+q_{T,c}^{\ell}+\tau(1), \ \ell \geq 1 \\
          p_{\rho}(1)&=&  -\sum_{\ell=1}^L\sum_{i\in I_{\ell}}p_{i}^{\ell}(1) (q_{T,i}^{\ell}-q_{T,c}^{\ell})^TR^T(1)\\
         p_{R}(1)& = &-\rho(1)\sum_{\ell=1}^{L}\sum_{i\in I_{\ell}}p_{i}^{\ell}(1) (q_{T,i}^{\ell}-q_{T,c}^{\ell})^T\\
    p_{\tau}(1) &=& -\sum_{\ell=1}^{L}\sum_{i\in I_{\ell}} p_{i}^{\ell}(1)
\end{array}\right.
\end{equation}
\end{prop}
\begin{rem}
  In the proposition and in the derivation of \eqref{endpoint_eq}, we encode $R$, $p_R$ and $r$ as $d\times d$ matrices and consider $H_{\Simd}$ as a function the variables $\bg\in \R^*\times\R^d\times \R^{d\times d}, \bpg= \R\times\R^d\times \R^{d\times d}$ and  $\bs\in \R\times\R^d\times \fsod$ with $\fsod\simeq \{r\in \R^{d\times d}\ |\ r+r^T=0\}$, in other words we consider the extension of the smooth action of $\Simd$ to $\R^*\times \R^{d\times d}\times \R^d$ that coincides with its restriction to $\Simd$.
\end{rem}
\begin{proof}
    The differential equations  $\dot{q}_{i,t}^{\ell} = u_t^{\ell}(q_{i,t})$ and $
     \dot{p}_{i,t}^{\ell} = -\left(du_t^{\ell}(q_{i,t}^{\ell})\right)^*p_{i,t}^{\ell}$ follow directly from the computation of the derivatives $\partial_p\mathcal{H}_{\bQ}$ and $\partial_q\mathcal{H}_{\bQ}$.  Furthermore, similarly to propositions \ref{Multiscale2} and \ref{control_multi2}, solving $\partial_\bu \mathcal{H}_{\bQ} = 0$ gives us $u^{\ell}(t) = \sum_{1\leq k\leq \ell}K^k\left(\sum_{L\geq m\geq k}\sum_{i\in I_m} \delta_{q_{i}^m(t)}^{p^m_i(t)}\right)$. 
The hamiltonian equations for the $\Simd$ part 
can be simply written as
$$
\left\{\begin{array}{l}
     (\dot{\rho},\dot{R},\dot{\tau})= (\alpha\rho, rR_T,\sigma)\\
      (\dot{p}_\rho,\dot{p_R},\dot{p_\tau}) = (-\alpha p_\rho,-r^Tp_{R}, 0) \\
      \alpha = \rho p_\rho,\ \sigma = p_{\tau}, \\
      r=(p_RR^T-Rp_R^T)/2
      \end{array}\right.
$$
One checks immediatly that  $\rho p_\rho$, $R^Tp_R$ and $p_\tau$ are conserved quantities during the dynamic (also a consequence of the right invariance of $H_{\Simd}$ and Noether's theorem) so that  $\alpha$, $\sigma$ and $R^TrR$ are conserved. In particular $\dot{R}(t)=R(t)r(0)=r(0)R(t)$ and $R(t)=\exp_{\Sod}(r(0)t)$. 
\\

We finish by proving equations \eqref{endpoint_eq}, which follows from the endpoint conditions. Indeed, for the constraint $g(\bm{q},\bg) = \frac{1}{2}\sum_{\ell=0}^{L}\sum_{i\in I_{\ell}} \left|q_i^{\ell} - \rho R(q_{T,i}^{\ell}-q_{T,c}^{\ell})-q_{T,c}^{\ell}-\tau\right|^2$, we have
\begin{equation}
    (\bm{p}(1),\bpg(1)) = - dg(\bm{q}(1),\bg(1))
\end{equation}
For $\ell\geq 1$, $i\in I_{\ell}$, and $\delta q_i^{\ell}\in \R^d$ we then have
\begin{align*}
    \langle p_{i,1}^{\ell},\delta q_i^{\ell}\rangle &= - \partial_{q_i^{\ell}}g(\bm{q}(1),\bg(1))\delta q_i^{\ell} \\
    &= - \langle q_{i}^\ell(1)- (\rho R)(1)(q_{T,i}^{\ell}-q_{T,c}^{\ell})-q_{T,c}^{\ell}-\tau(1) , \delta q_i^{\ell} \rangle
\end{align*}
and therefore $p_{i}^{\ell}(1) = - q_{i}^\ell(1) + (\rho R)(1)(q_{T,i}^{\ell}-q_{T,c}^{\ell})+q_{T,c}^{\ell}+ \tau(1)$. 
\begin{align*}
    \langle p_{R}(1),\delta R\rangle  &= - \partial_{R}g(\bm{q}_1,\tau_1,R_1)\delta R \\
    &= - \sum_{\ell=1}^L\sum_{i\in I_{\ell}} \langle  q_{i}^\ell(1)- (\rho R)(1)(q_{T,i}^{\ell}-q_{T,c}^{\ell})-q_{T,c}^{\ell}-\tau(1) , -\rho \delta R (q_{T,i}^{\ell}-q_{T,c}^{\ell}) \rangle \\
    &= -\rho(1)\sum_{\ell=1}^L\sum_{i\in I_{\ell}} \langle p_{i}^{\ell}(1),\delta R (q_{T,i}^{\ell}-q_{T,c}^{\ell})\rangle \\
    &= -\rho(1)\langle \sum_{\ell=1}^L\sum_{i\in I_{\ell}}p_{i}^{\ell}(1) (q_{T,i}^{\ell}-q_{T,c}^{\ell})^T, \delta R\rangle
\end{align*}
Similarly, we get $ p_{\rho}(1)=  -\sum_{\ell=1}^L\sum_{i\in I_{\ell}}p_{i}^{\ell}(1) (q_{T,i}^{\ell}-q_{T,c}^{\ell})^TR^T(1)$ and 
$p_\tau(1)= -\sum_{\ell=1}^L\sum_{i\in I_{\ell}} p_{i}^{\ell}(1)$
Hence the proposition.
\end{proof}

\section{Euler-Poincaré equations}
In this part, we are back to the general (GGA) framework of a graded group structure acting on a Banach manifold introduced in section \ref{sec:sub-riem}~:
We study the Hamiltonian equations for the normal sub-Riemannian geodesics. In particular, the lift of the geodesics on the groups $G^{k_0}$ satisfies an integrated version of the Euler-Poincaré equations, and we prove this equation admits a unique and global solution. This means the geodesics on the Banach space $\mathcal{Q}$ is totally determined by the geodesics in the groups. In practice, this allows us to describe shapes with different spaces while keeping the same geodesics (cf. corollary \ref{cor:19.6.1}), and we develop some examples of applications in the next section. 

\subsection{Induced metrics on Banach manifold}
We recall that $V$ is a Hilbert space continuously included in $T_eG^{k_0+2}$ and that the map $K_V : V^*\rightarrow V$ denotes the inverse of the Riesz isometry map on $V$. We consider again the $C^2$ Hamiltonian given by :
$$
\mathcal{H}_{\mathcal{Q}} : 
\left|
  \begin{array}{ccl}
    T\mathcal{Q}^* \times V & \longrightarrow & \R \\
    (q, \ p, \ u) &\longmapsto & (p\,|\,u\cdot q) - \frac{1}{2}|u|_V^2 \\
  \end{array}
\right.
$$
In this section we are interested on the normal Hamiltonian equations :
\begin{equation}
\label{normHamEq}
    \left\{
\begin{array}{ll}
 \left(\dot{q}, \ \dot{p}\right)   =  \nabla^{\omega}\mathcal{H}_{\mathcal{Q}}(q,p,u) \\
     \partial_u\mathcal{H}(q,p,u)=0
\end{array}
\right.
\end{equation}

We will need the momentum map associated to a couple $(q,p)\in T\mathcal{Q}^*$ :
$$
m_{\mathcal{Q}} : \left|
  \begin{array}{ccl}
    T\mathcal{Q}^* & \longrightarrow & (T_eG^{k_0})^* \\
    (q, \ p) &\longmapsto & m_{\mathcal{Q}}(q,p) \\
  \end{array}
\right.
$$
where for $v\in T_eG^{k_0}$, $$\left(m_{\mathcal{Q}}(q, p)\,|\,v\right) = (p\,|\,\xi_q(v)) = (p\,|\,v\cdot q)\,.$$
By (S.2), for $(q,p)\in T\mathcal{Q}^*$, the infinitesimal action $\xi_q : T_eG^{k_0}\rightarrow T\mathcal{Q}$ is continuous, so that
$$m_\mathcal{Q}(q,p)=\xi_q^*(p)$$
and for $l\geq 0$ the momentum map $m_{\mathcal{Q}}(q,p)$ restricts to a continuous form in $(T_eG^{k_0+l})^*$ and $V^*$. 
The last condition of equation \eqref{normHamEq}
$$
\partial_u\mathcal{H}(q,p,u) = 0 = (p|\xi_q(.))-\langle u,\cdot\rangle_V
$$
has a unique solution $u(q,p) = K_V \xi_q^*(p)=K_Vm_{\mathcal{Q}}(q,p)$, where to apply $K_V$ we are doing a slight abuse of notation by considering $m_{\mathcal{Q}}(q,p)$ as an element of $V^*$ through its continuous restriction to $V$.
Since $\mathcal{H}$ is strictly concave in $u$, we also have that $u(q,p) = \max_{u\in V}\mathcal{H}(q,p,u)$. As the restriction $\xi : V \times \mathcal{Q}\rightarrow T\mathcal{Q}$ is a $C^2$ mapping (S.3), the mapping $q\in\mathcal{Q}\mapsto {\xi_q}_{|V} \in \mathcal{L}(V,T_q\mathcal{Q})$ is $C^1$ with locally lipschitz derivative, and therefore $(q,p)\mapsto \xi_q^*(p)_{|V} \in V^*$ is $C^1$ with locally lipschitz derivative. Finally the mapping
$$
u : \left|
  \begin{array}{ccl}
    T\mathcal{Q}^* & \longrightarrow & V \\
    (q,p) &\longmapsto & K_V \xi_q^*(p) \\
  \end{array}
\right.
$$
is $C^1$ with locally lipschitz derivative. This allows to introduce the reduced Hamiltonian given by :
$$
h(q,p) = \mathcal{H}(q,p,u(q,p)) = \frac{1}{2} |u(q,p)|_V^2
$$ which is also $C^1$ with locally lipschitz derivative by composition. Its symplectic gradient is given in charts by 
$$
\nabla^{\omega}h(q,p) = \left(\partial_p h(q,p), -\partial_q h(q,p) \right) = \nabla^{\omega}\mathcal{H}(q,p,u(q,p)),
$$
so that equation \eqref{normHamEq} reduces to solving
\begin{equation}
\label{redHamEq}
 \left(\dot{q}, \ \dot{p}\right)   =  \nabla^{\omega}h(q,p) 
\end{equation}
From this equation we then retrieve the velocity $u_t = K_V \xi_{q_t}^*(p_t)$. We get the following theorems
under the (GGA) framework.
\begin{prop}[Existence and uniqueness of the Hamiltonian flow in $\mathcal{Q}$]
\label{prop:19.6.1}
We have global existence and uniqueness of the geodesic equation :
\begin{equation}
\label{redHamEq2}
\left\{
\begin{array}{ll}
    \dot{q} \ = \partial_{p} h(q,p) \\
     \dot{p} = -\partial_{q}h(q,p)
\end{array}
\right.
\end{equation}
with initial value $(q_0,p_0) \in T\mathcal{Q}^*$. This Hamiltonian flow $t \mapsto (q_t, p_t)$ is $C^1$ on $T\mathcal{Q}^*$.
\end{prop}
This Hamiltonian flow can be uniquely lifted as an Hamiltonian flow in the space $T^*G^{k_0+1}$ :
\begin{theo}[Lifted trajectory and integrated Euler-Poincaré equations]
\label{theo:int_epdiff}
Let $t\mapsto (q_t,p_t)\in C^1(I,T\mathcal{Q}^*)$ be a Hamiltonian flow of \eqref{redHamEq2} and let $t\mapsto m_t\doteq m_{\mathcal{Q}}(q_t,p_t)$ be the associated momentum trajectory in $(T_eG^{k_0})^*$. Then:
\begin{enumerate}
\item there exists a trajectory $g \in C^1(I,G^{k_0+1})$ solution of 
$\dot{g}_t=(K_V m_t)\cdot g_t$ and satisfying $q_t=g_t\cdot q_0$ and hereafter called \emph{lifted trajectory};
\item the momentum map trajectory verifies
\begin{equation}
\label{int_epdiff2}
    m_t = \operatorname{Ad}_{g_t^{-1}}^*(m_0)\,;
\end{equation}
\item the lifted trajectory $t\mapsto g_t$ is uniquely defined as the solution in $G^{k_0+1}$ starting at $e_{G^{k_0+1}}$ of the autonomous differential system
\begin{equation}
\label{int_epdiff}
    \dot{g} = G_{m_0}(g)\doteq K_V\operatorname{Ad}_{g^{-1}}^*(m_0)\cdot g
\end{equation}
with $G_{m_0}:G^{k_0+1}\to TG^{k_0+1}$ locally Lipschitz.
\end{enumerate}
\end{theo}

\begin{rem}
    The equation \eqref{int_epdiff2} is well defined since $m_0\in (T_eG^{k_0})^*$, the adjoint map $\operatorname{Ad}_{g^{-1}} : T_eG^{k_0}\rightarrow T_eG^k{k_0}$ is also continuous. This equation is an integrated form of the usual Euler-Poincaré equations for right invariant metric on Lie groups.
\end{rem}

\begin{cor}[Uniqueness of the momentum map trajectory]
\label{cor:19.6.1}
    Let $\mathcal{Q}_1$ and $\mathcal{Q}_2$ be Banach manifolds such that $G^{k_0}$ acts on $\mathcal{Q}_1$ and $\mathcal{Q}_2$ within the (GGA) framework. Let $(q^1_0,p^1_0)\in T\mathcal{Q}_1^*$ and $(q^2_0,p^2_0)\in T\mathcal{Q}_2^* $ such that $m_{\mathcal{Q}_1}(q^1_0,p^1_0) = m_{\mathcal{Q}_2}(q^2_0,p^2_0)$. Then for all $t\geq 0$, we have 
    $$
    m_t\doteq m_{\mathcal{Q}_1}(q^1_t,p^1_t) = m_{\mathcal{Q}_2}(q^2_t,p^2_t)$$
    and there is a unique lifted trajectory $t\mapsto g_t$ satisfying 
     $$
    \dot{g}_t = K_Vm_t\cdot g_t
    \mbox{ and } \left\{ \begin{array}{l}
    q^1_t=g_t\cdot q^1_0\\
    \\
    q^2_t=g_t\cdot q^2_0
    \end{array}\right.$$
\end{cor}
\begin{rem}
    Following remark \ref{GGAonG}, we can study the case where $\mathcal{Q}_2=G^{k_0}$ so that we come back to a simple action of $G^{k_0}$ on $\mathcal{Q}_1=\mathcal{Q}$ within the (GGA) framework. The momentum map is simply $m_t = p_t^{G^{k_0}}\cdot g_t^{-1}$ and we obtain equivalence between the normal geodesics in the group $G^{k_0}$ and the normal geodesics in the shape space $\mathcal{Q}$.
\end{rem}

\subsection{Proofs of proposition \ref{prop:19.6.1}, theorem \ref{theo:int_epdiff} and corollary \ref{cor:19.6.1}}
 We first start by proving the local existence and uniqueness of solutions of Hamiltonian equations :
 
 \begin{lemma}
We have local existence and uniqueness of the geodesic equation :
\begin{equation}
\left\{
\begin{array}{ll}
    \dot{q} \ = \partial_{p} h(q,p) \\
     \dot{p} = -\partial_{q}h(q,p)
\end{array}
\right.
\end{equation}
with initial value $(q_0,p_0) \in T\mathcal{Q}^*$. This Hamiltonian flow $t \mapsto (q_t, p_t)$ is $C^1$ on $ T\mathcal{Q}^*$ and we have $u_t = K_Vm_Q(q_t,p_t)$
 \end{lemma}

 \begin{proof}
The reduced Hamiltonian $h$ is $C^1$ with locally lipschitz derivative and thus its symplectic gradient is locally lipschitz. By Picard-Lindelof theorem, we can therefore integrate the equation$~$:
 $$
(\dot{q}, \dot{p})=\nabla^{\omega}h(q,p)
 $$
 and get for any initial data $(q_0,p_0) \in T\mathcal{Q}^*$ a unique Hamiltonian $C^1$ flow $t\mapsto(q_t, p_t)$. 
 \end{proof}

For the rest of this part, we fix $t\mapsto (q_t,p_t)$ a maximal solution of the Hamiltonian equations, with initial value $(q_0,p_0)$. We denote by $J$ the interval of definition of this maximal solution. To show that $J=\R$, we will first need to lift the flow on the group $G^{k_0}$, and study its convergence.
We first define the momentum associated to the flow 
$$
m_t = m_{\mathcal{Q}}(q_t,p_t)
$$
and we now determine the regularity of the momentum 
\begin{lemma}[Regularity of the momentum map]
\label{reg_mom}
For all $t$, $m_t\in (T_eG^{k_0})^*$, and the application $t\mapsto m_t$ is $C^1$ on $(T_eG^{k_0+2})^*$
\end{lemma}
\begin{proof}
By definition, we have $m_Q(q,p) = p\circ \xi_q = \xi_q^*(p)$. By hypothesis (S.2), for all $t\in \R$, the action $\xi_{q_t}:T_eG^{k_0}\rightarrow T_{q_t}\mathcal{Q}$ is continuous, and as $p_t \in T_{q_t}\mathcal{Q}^*$, the momentum $m_t = m_Q(q_t, p_t)$ induced a continuous form on $T_eG^{k_0}$.

Furthermore, by (S.3), the restriction $\xi : T_eG^{k_0+2}\times \mathcal{Q}  \rightarrow  T\mathcal{Q}$ is $C^2$. Therefore the mapping $q \in \mathcal{Q} \mapsto \xi_q \in \mathcal{L}(T_eG^{k_0+2},T_q\mathcal{Q})$ is $C^1$ \cite[Theorem 5.3]{hideki1997infinite}. Consequently the dual mapping $q \in \mathcal{Q} \mapsto \xi_q^* \in \mathcal{L}(T_q\mathcal{Q}^*,(T_eG^{k_0+2})^*)$ is also $C^1$.
As $m_t=\xi_{q_t}^*(p_t)$ and $p_t$ and $q_t$ are $C^1$, the momentum $t\mapsto m_t$ is $C^1$ in $(T_eG^{k_0+2})^*
$ (and similarly on $V^*$ since $V$ is continuously embedded in $T_eG^{k_0+2}$) \end{proof}

We also lift the Hamiltonian flow on the group. We first retrieve the eulerian velocity $u_t=K_V m_t$~:

\begin{lemma}
The equation \begin{equation}
    \dot{g}_t = u_t\cdot g_t
    \label{relev}
\end{equation} is well defined on $G^{k_0+1}$ and has a unique maximal solution $C^1$ on $G^{k_0+1}$, defined on $J$. The solution $g_t$ is such that $q_t=g_t\cdot q_0$
\end{lemma}

\begin{proof}
This is a direct corollary of Proposition \ref{reg_groups} and Lemma \ref{reg_mom}.
We saw that $t\mapsto m_t = m_{\mathcal{Q}}(q_t, p_t)$ is $C^1$ in $(T_eG^{k_0+2})^*$.  Moreover, since :
$$
(T_eG^{k_0+2})^*\stackrel{i_V^*}\rightarrow V^* \stackrel{K_V}\simeq V \hookrightarrow T_eG^{k_0+2}
$$
the velocity $t\mapsto u_t=K_Vm_t$ is $C^1$ in $T_eG^{k_0+2}$. By proposition \ref{reg_groups}, there exists $g_t \in AC_{L^1}(I,G^{k_0+1})$ satisfying equation \eqref{relev}. Now, since $t\mapsto u_t$ is $C^1$ and the product $$
  \begin{array}{ccl}
    T_eG^{k_0+2}\times G^{k_0+1} & \longrightarrow & TG^{k_0+1} \\
    (u, \ g) &\longmapsto & u\cdot g \\
  \end{array}
$$ is $C^1$ (G.4), and since $g_t$ is continuous, $g$ is therefore $C^1$. 
\end{proof}

We can now prove that $J=\R$ :

\begin{lemma}
    The Hamiltonian flow is defined globally.
\end{lemma}
\begin{proof}
\emph{Step 1 : $|u_t|$ is constant. }
To prove that $g_t$ is defined on $J$, we first show that $|u_t|$ is constant. For $t\in J$, we have
\begin{align*}
    \frac{d}{dt} h(q_t,p_t) &= \partial_q h(q_t,p_t)(\dot{q}_t) + \partial_p h(q_t,p_t)(\dot{p}_t) \\
    &= \partial_q h(q_t,p_t)(\partial_p h(q_t,p_t)) + \partial_p h(q_t,p_t)(-\partial_q h(q_t,p_t)) \\
    &= 0\,.
\end{align*}
But we also have $h(q_t,p_t)= \mathcal{H}_{\mathcal{Q}}(q_t,p_t,u_t) = \frac{1}{2}|u_t|_V^2$, therefore $u_t$ is constant.
    
\emph{Step 2 : $J=\R$}

Let $b=\sup J$, and suppose $b<+\infty$. We are going to prove that in that case, the Hamiltonian flow $(q_t,p_t)$ converges and thus a contradiction. 

We recall that by (G.4), $(g,u)\in G^{k_0+1}\times T_eG^{k_0+1}\mapsto T_eR_g u \in TG^{k_0+1}$ is continuous. Then by \cite[Theorem 5.3]{hideki1997infinite}, the mapping $g\in G^{k_0+1}\rightarrow T_e R_g \in \mathcal{L}(T_eG^{k_0+1},TG^{k_0+1})$ is also locally bounded. Let $(U,\phi)$ a local chart of $G^{k_0+1}$ around $e$, and suppose that there exists $K>0$, such that for all $g\in U$
$$
\lVert d_g\phi\circ T_e R_g \rVert < K\,.
$$
Such an open set U always exists since, if we identify $T_eG^{k_0+1}$ and the Banach space $\mathbb{B}^{k_0+1}$, the map $g\in G^{k_0+1}\mapsto d_g\phi\circ T_e R_g \in \mathcal{L}(T_eG^{k_0+1},T_eG^{k_0+1})$ is locally bounded. Let $t_0\in J$, and by translating the curve $g_t$ we define $\tilde{g}_t\in G^{k_0+1}$ :
$$
\tilde{g}_t = g_tg_{t_0}^{-1} = R_{g_{t_0}^{-1}}g_t\,.
$$
The curve $\tilde{g_t}$ verifies the Cauchy problem :
\begin{equation}
    \left\{\begin{array}{l}
        \dot{\tilde{g}}_t=u_t \cdot \tilde{g}_t=T_eR_{\tilde{g}_t}(u_t)\,,   \\
         \tilde{g}_0=e \,.
    \end{array}
    \right.
\end{equation}
On the open subset $\tilde{J}=\{t \in J,\tilde{g}_t\in U \}$, we can therefore look at the curve $\alpha_t=\phi(\tilde{g}_t)$ on the local chart $U$, whose derivative is given by 
$\dot{\alpha}_t = d_{\tilde{g}_t}\phi\circ T_e R_{\tilde{g}_t}(u_t)$. Therefore, for all $t\in\tilde{J}$, 
\begin{align*}
    |\dot{\alpha}_t|_{T_eG^{k_0+1}} &= |d_{\tilde{g}_t}\phi\circ T_e R_{\tilde{g}_t}(u_t)|_{T_eG^{k_0+1}}\\
    &\leq K |u_t|_{T_eG^{k_0+1}} \\
    &\leq K' |u_t|_V \ \ \mbox{as }V\hookrightarrow T_eG^{k_0+1}
\end{align*}
and as $u_t$ is bounded, by \cite{lang2001fundamentals} (proposition 1.1, p.68), there exists $a,\epsilon>0$ independent of $t_0$ so that the flow $\alpha_t$ can be defined on $]t_0-\epsilon,t_0+\epsilon[\cap J$ and $\alpha_t\in\phi(U)$. Therefore, we can take $t_0\in]b-\epsilon,b[$, the curve $\alpha_t$ will also be defined on $]b-\epsilon,b[$, is Lipschitz and hence converges towards $\alpha_b\in\phi(U)$. Thus $g_t$ converges towards $g_b\in G^{k_0+1}$ when $t\rightarrow b$. By (S.1), the action $A$ of $G^{k_0+1}$ on $\mathcal{Q}$ is continuous, and thus $q_t = A(g_t,q_0)$ converges towards $q_b=A(g_b,q_0)\in \mathcal{Q}$.

 We also have that $p_t$ is solution of the linear differential equation :
 $$
 \dot{p}_t = -\partial_q h(q_t,p_t) = -\left(\partial_q (\xi_{q}(u_t))_{q=q_t}\right)^*p_t
 $$
The mapping $(q,u) \mapsto (\partial_q\xi_q u)^*$ is continuous and linear with regards to $u$, therefore there exists $L>0$ and $\delta>0$ such that in a local chart around $q_b$, and for $|u|<\delta$, we have :
$$
|(\partial_q\xi_q u)^*|_{\mathcal{L}(\mathbb{Q}^*,\mathbb{Q}^*)} \leq L
$$
where $\mathbb{Q}$ is a Banach model space for $\mathcal{Q}$, and where we identify $((\partial_q\xi_q u)^*)$ with its induced mapping in the chart around $q_b$.
Therefore, for $t$ close to $b$, we also have :
$$
|(\partial_{q}\xi_{q} u_t)_{q=q_t}^*|_{\mathcal{L}(\mathbb{Q}^*,\mathbb{Q}^*)} \leq L|u_t|/\delta=L|u_0|/\delta
$$
Therefore by Gronwall lemma, $p$ is bounded when $t \in ]b-\epsilon',b[$ with $\epsilon'>0$, and  thus $\dot{p}$ is also bounded. This implies again that $p_t$ converges in $T\mathcal{Q}^*$ towards $p_b$. 

Therefore the Hamiltonian flow $(q_t,p_t)$ converges when $t\rightarrow b$, and then $b=+\infty$.
A similar proof would show that we also have $\inf J = -\infty$, hence $J=\R$
\end{proof}

\begin{lemma}
    The momentum $m_t$ is solution of an integral version of EPdiff, i.e., for $v\in T_eG^{k_0}$~:
    \begin{equation}
        \frac{d}{dt}\left(m_t\,|\,\operatorname{Ad}_{g_t}(v)\right) = 0
    \end{equation}
\end{lemma}

\begin{proof}
Let $v \in T_eG^{k_0}$, we have $\left(m_t\,|\,\operatorname{Ad}_{g_t}(v)\right) = \left(p_t\,|\,\mbox{Ad}_{g_t}(v)\cdot q_t\right)$. We will make all computations in canonical coordinates on $T^*\mathcal{Q}$ near $(q_t,p_t)$ (i.e. we locally identify $TQ^*$ with the trivial bundle $E\times E^*$ where $\varphi:U\to E$ is a local chart around $(q_t,p_t)$ for $t=t_0$), and therefore we have $$\frac{d}{dt}\left(p_t\,|\,\mbox{Ad}_{g_t}(v)\cdot q_t\right) =\left(p_t\,|\,\frac{d}{dt}\left(\mbox{Ad}_{g_t}(v)\cdot q_t\right)\right) + \left(\dot{p}_t\,|\,\mbox{Ad}_{g_t}(v)\cdot q_t\right)$$

where $\frac{d}{dt}\left(\mbox{Ad}_{g_t}(v)\cdot q_t\right) \in T_q\mathcal{Q}$ in the charts.
\def\tA{\tilde{A}}
\def\inte{\operatorname{int}}
\def\Ad{\operatorname{Ad}}
We will consider in this proof the mapping
$$
\tilde{A} : \left|
  \begin{array}{ccl}
    G^{k_0+2} \times G^{k_0} & \longrightarrow & \mathcal{Q} \\
    (g, \ g') &\longmapsto & A_{q_0}(gg')=gg'\cdot q_0 \\
  \end{array}
\right.
$$
The mapping $\tilde{A}$ is $C^2$ as $A_{q_0}$ is $C^{\infty}$ (S.2) and by (G.3).
Let us first notice that we have the three equalities
\begin{equation*}
     \tA(g,g') =\left\{
     \begin{array}[h]{ll}
       A_{q_0}\circ L_{gg_t}\circ L_{g_t^{-1}}(g') &(Ea)\\
      A(\inte_{gg_t}\circ L_{g_t^{-1}}(g'),g\cdot q_t)) &(Eb)\\
     A(g,A(\inte_{g_t}\circ L_{g_t^{-1}}(g'),q_t))&(Ec)\\
     \end{array}\right.
\end{equation*}
  We start first by computing the derivative $\partial_1 \partial_2 \tilde{A}(e,g_t) \big(u_t, g_tv\big)$. We have, for $g\in G^{k_0+2}$:
  \begin{align*}
   \partial_2 \tilde{A}(g,g_t) T_eL_{g_t}(v)
    &\overset{(Eb)}{=} \partial_1 A(e,g\cdot q_t)\Ad_{gg_t}(v)\\
    &= \xi_{g\cdot q_t}(\Ad_{gg_t}(v))  
  \end{align*}
  and using $(Ea)$, we get also that $ \partial_2 \tilde{A}(e,g_t) T_eL_{g_t}(v)=T_{gg_t}A_{q_0}\circ T_eL_{gg_t}(v)$.  As $A_{q_0}$ is $C^{\infty}$ on $G^{k_0}$, and $g\in G^{k_0+1}\mapsto T_e L_{gg_t}(v)\in TG^{k_0}$ \ref{smooth_op} is also $C^{\infty}$, then the application $g\in G^{k_0+1}\mapsto \partial_2 \tilde{A}(g,g_t) T_eL_{g_t}(v)$ is $C^1$ on $G^{k_0+1}$. Furthermore, since the embedding $G^{k_0+2}\hookrightarrow G^{k_0+1}$ is smooth (G.1) and since $s\mapsto g_{s+t}g_{t}^{-1}$ is 
  $C^1$ in $G^{k_0+1}$ with derivative $u_t$ at $s=0$
  \begin{equation}
    \label{eq:12.12.4}
    \begin{split}
      \partial_{1,2}\tA(e,g_t) (u_t,g_tv)& = \partial_g(\xi_{g\cdot q_t}(\Ad_{gg_t}(v)))_{|g=e}(u_t)=\frac{d}{ds}(\xi_{q_{t+s}}(\Ad_{g_{t+s}}(v)))_{|s=0}\\
      &=\frac{d}{dt}(\xi_{q_{t}}(\Ad_{g_{t}}(v)))
  \end{split}
  \end{equation}



By Schwarz theorem, we also have
\begin{align*}
    \partial_{1,2} \tilde{A}(e,g_t) (u_t, g_tv) 
    &= \partial_{2,1}\tilde{A}(e,g_t) (u_t, g_tv) \\
    &\overset{(Ec)}{=} \partial_{2,1}A(e,q_t) \bigg(u_t,\partial_1A(e,q_t)   \Ad_{g_t}\big(T_{g_t}L_{g_t^{-1}}\circ T_eL_{g_t}(v)\big)\bigg)\\
  &= \partial_{2,1}A(e,q_t) (u_t,\Ad_{g_t}(v)\cdot q_t)\\
    &= \partial_{q} (\xi_{q}(u_t))_{|q=q_t} (\Ad_{g_t}(v)\cdot q_t)
\end{align*}

Therefore we have
\begin{equation}
        \label{term1}
        \left(p_t\,|\,\frac{d}{dt}\left(\Ad_{g_t}(v)\cdot q_t\right)\right) = \left(p_t\,|\,\partial_{q} (\xi_{q}(u_t))_{|q=q_t} (\Ad_{g_t}(v)\cdot q_t) \right)
      \end{equation}
      Since $p_t$ is solution of Hamiltonian equation, we get for $\delta q \in T_{q_t}\mathcal{Q}$ :
      $$\left(\dot{p}_t\,|\,\delta q\right) = - \partial_q\mathcal{H}(q_t,p_t,u_t) \delta q = \left(p_t\,|\,\partial_q (\xi_{q}(u_t))_{|q=q_t} \delta q\right)\,.$$

      Now for $\delta q= \mbox{Ad}_{g_t}(v)\cdot q_t$ :
\begin{equation}
    \label{term2}
    \left(\dot{p}_t\,|\,\mbox{Ad}_{g_t}(v)\cdot q_t\right) = -\left(p_t\,|\,\partial_q (\xi_{q}(u_t))_{|q=q_t}(\mbox{Ad}_{g_t}(v)\cdot q_t)\right)
\end{equation}

By summing \ref{term1} and \ref{term2} we therefore have $\frac{d}{dt}\left(p_t\,|\,\mbox{Ad}_{g_t}(v)\cdot q_t\right) = 0$ which concludes the proof of the lemma.
\end{proof}

Now we can complete the proof of the proposition. We have proved that for all $v\in V$, $t\mapsto(m_t\,|\,\operatorname{Ad}_{g_t}(v))$ is constant, i.e. :
\begin{equation*}
m_t=\operatorname{Ad}^*_{g^{-1}_t}(m_0)
\end{equation*}
We define 
$$
G_{m_0}: \left|
  \begin{array}{ccl}
    G^{k_0+1} & \longrightarrow & TG^{k_0+1} \\
    g &\longmapsto & K_V\operatorname{Ad}^*_{g^{-1}}(m_0)\cdot g=T_eR_g(K_V\operatorname{Ad}^*_{g^{-1}}(m_0)) \\
  \end{array}
\right.
$$
Since $\dot{g}_t=u_t\cdot g_t = K_Vm_t\cdot g_t$, the flow $g_t$ is therefore solution of the ordinary equation 

\begin{equation}
    \label{fin_eq}
        \dot{g}_t=G_{m_0}(g_t)
    \end{equation}
 on Banach space $G^{k_0+1}$, and therefore theorem \ref{theo:int_epdiff} is proved. To complete the proof of corollary \ref{cor:19.6.1}, we prove this equation admits one unique solution. We have the following lemma

\begin{lemma} 
\label{lem:fin_eq}
Let $m_0 \in (T_eG^{k_0})^*$. Then the mapping $G_{m_0}$ is locally-lipschitz from $G^{k_0+1}$ to $TG^{k_0+1}$ and the equation \ref{fin_eq} admits a unique maximal solution
\end{lemma}
\begin{proof}We start with the first point.
    We prove that $G_{m_0}$ is a locally-Lipschitz mapping. We have that for $g\in G^{k_0+1}$, $$\operatorname{Ad}_{g^{-1}} =  T_gL_{g^{-1}} \circ T_e R_g = (T_eL_g)^{-1} \circ T_e R_g$$ By proposition \ref{smooth_op}, the mapping $g \in G^{k_0+1}\mapsto T_eL_g\in \mathcal{L}(T_eG^{k_0},T_gG^{k_0})$ is locally-Lipschitz. In a local chart of the bundle $TG^k_0\rightarrow G^{k_0+1}$ around $g\in G^{k_0+1}$, $g\mapsto T_eL_g$ is the mapping $g\mapsto (g, \widehat{T_eL_g}) \in V\times GL(T_eG^{k_0})$. As $GL(T_eG^{k_0})$ is a Banach Lie group, the inverse mapping $\operatorname{inv} : GL(T_eG^{k_0}) \rightarrow GL(T_eG^{k_0})$ is smooth. Thus, the mapping $g \in G^{k_0+1}\mapsto (T_eL_g)^{-1}\in \mathcal{L}(T_gG^{k_0},T_eG^{k_0})$ is locally-Lipschitz. Furthermore, by (G.4), the mapping $(g,u)\in G^{k_0+1}\times T_eG^{k_0+2}\mapsto T_eR_g u \in TG^{k_0}$ is $C^2$, and thus by \cite[Theorem 5.3]{hideki1997infinite}, $g\in G^{k_0+1} \mapsto T_eR_g \in \mathcal{L}(T_eG^{k_0+2},T_gG^k_0)$ is $C^1$ and thus locally-Lipschitz. By bilinearity of the composition on 
    the product $\mathcal{L}(T_gG^k_0,T_eG^k_0)\times \mathcal{L}(T_eG^{k_0+2},T_gG^k_0)$, the mapping $g\in G^{k_0+1} \mapsto \operatorname{Ad}_{g^{-1}} \in \mathcal{L}(T_eG^{k_0+2},T_eG^k)$ and then dual mapping $g\in G^{k_0+1} \mapsto \operatorname{Ad}^*_{g^{-1}} \in L\left((T_eG^{k_0})^*,(T_eG^{k_0+2})^*\right)$ are also locally Lipschitz. Therefore  
    $$
  \begin{array}{ccl}
    G^{k_0+1} & \longrightarrow & T_eG^{k_0+2} \\
    g &\longmapsto & K_V\operatorname{Ad}^*_{g^{-1}}(m_0) \\
  \end{array}
    $$
    is locally Lipschitz. Now, since $\xi : T_eG^{k_0+2} \times G^{k+1}\mapsto TG^{k_0+1}$ is $C^1$, then $G_{m_0}$ is locally Lipschitz, and by Picard-Lindelof, the equation \eqref{fin_eq} has a unique solution.
\end{proof}

We can now finish the proof of corollary \ref{cor:19.6.1}. Let $((q^1_t,p^1_t),u^1_t) \in T\mathcal{Q}_1^*\times V$ and $((q^2_t,p^2_t),u^2_t) \in T\mathcal{Q}_2^*\times V$ satisfying the conditions of corollary \ref{cor:19.6.1}, and let $g^1_t$ and $g^2_t$ be their lifts in $G^{k_0+1}$. Then both $g^1_t$ and $g^2_t$ satisfies the differential equation in $G^{k_0+1}$ 
$$
\dot{g}_t = G_{m_0}(g_t)
$$
with $g^1_0=g^2_0=e_{G^{k_0+1}}$. By (G.5) and lemma \ref{lem:fin_eq}, this equation thus admits a unique solution, and therefore $g^1_t=g^2_t$

\section{Extended diffeomorphisms groups and regularity of the coadjoint map}
The usual path on riemannian shape spaces defined from the induced metric by the action of diffeomorphisms has been following the \emph{reduction} route: one computes the hamiltonian on the shape space $\mathcal{Q}$, derives there shooting algorithms encoded from a given initial momenta in the cotangent space and when needed the lifted trajectory. This happens to be quite effective when the space space $\mathcal{Q}$ supports some sort of discretisation into a finite dimensional space. A common setting is the case of landmarks where spaces are described as point clouds or various triangulations. If the reduction route can be beneficial from a memory point of view since providing finite dimensional encoding of the lifted geodesic in the group of diffeomorphisms, interestingly over-parametrization, or the \emph{unreduction} route, can be also helpful since it may allow a gain of regularity in the representation of the momentum compare to the reduced case. We consider in the following section an illustration of this phenomenon for the group of diffeomorphisms arising in a common situation of inexact matching for transport of atlases.
\subsection{Some extensions of groups of diffeomorphisms}
In this part, we define extensions of groups of diffeomorphisms that satisfy the conditions (G.1-5).\\

We first recall the definition of the diffeomorphisms group $\operatorname{Diff}_{C_0^k}(\R^d)$  for $k\geq 1$:
$$
G^k = \operatorname{Diff}_{C_0^k}(\R^d)=\left(\id + C_0^k(\R^d,\R^d)\right)\cap\operatorname{Diff}^1(\R^d)\,.
$$

Now we introduce two semi-direct products, $H^k$ and $N^k$: 
$$
H^k = \operatorname{Diff}_{C_0^k} \ltimes C_0^{k-1}(\R^d,\R^{+}_*)
$$
with group operation given for all $(\varphi,\omega),(\varphi',\omega')\in H^k$ by 
$$
(\varphi,\omega)\cdot (\varphi',\omega') = (\varphi\circ\varphi',\omega\circ\varphi' \omega')
$$
and  
$$
N^k = \operatorname{Diff}_{C_0^k} \ltimes C_0^{k-1}(\R^d,GL_d)
$$
with group operation given for all $(\varphi,A),(\varphi',A')\in N^k$  by 
$$
(\varphi,A)\cdot (\varphi',A') = (\varphi\circ\varphi',A\circ\varphi' A')
$$
where $GL_d$ is the general linear group on $\mathbb{R}$.
\begin{rem}
Clearly, $H^k$ and $N^k$ are (normal) extensions of the group $G^k$  so that $G^k$ can be seen as a subgroup of bigger groups (providing a natural setting for over-parametrization).
\end{rem}
\begin{rem}
  Similarly to $G^k$, we can consider the multiscale versions :
$$
\bH^k = \prod_{0\leq l< L} H^k, \ \bN^k = \prod_{0\leq l< L} N^k\,.
$$  
\end{rem}
In the following, for $\varphi\in G^k$, $J\varphi(x)\in GL_d$ is its Jacobian matrix at $x\in \R^q$ and $|J\varphi|(x)$ its Jacobian determinant.
\begin{prop} We have the following properties :
    \begin{enumerate}
        \item The families of groups $(H^k)_{k\geq 1}$, $(N^k)_{k\geq 1}$ satisfy the conditions (G.1-5).
        \item The group $G^k$ is embedded in both $H^k$ and $N^k$ through the \textbf{smooth} mappings
$$
\begin{array}{ccc}
i_H : \left|
  \begin{array}{ccl}
    G^k & \longrightarrow & H^k \\
    \varphi &\longmapsto & \left(\varphi,|J\varphi|\right) \\
  \end{array}\right.
     & \mbox{ and } &
     
  i_N:\left|
  \begin{array}{ccl}
    G^k & \longrightarrow & N^k \\
    \varphi &\longmapsto & (\varphi,J\varphi) \\
  \end{array}\right.
\end{array}
$$
and the differentials of this mappings are given by :
$$
\begin{array}{l}
  Ti_H : \left|\begin{array}{ccl}
    TG^k & \longrightarrow & TH^k \\
    (\varphi,\delta\varphi) &\longmapsto & \left((\varphi,|J\varphi|),(\delta\varphi,\ \operatorname{div}\left(\delta\varphi\circ\varphi^{-1}\right)\circ\varphi\left|J\varphi\right|)\right) \\ 
  \end{array}\right. \\ \\
       Ti_N : \left| \begin{array}{ccl}
    TG^k & \longrightarrow & TN^k \\
    (\varphi,\delta\varphi) &\longmapsto & \left((\varphi,|J\varphi|),(\delta\varphi,J\delta\varphi)\right) \\
  \end{array}\right.
\end{array}\,.
$$
 \item The mapping  $i_H$ is a morphism of groups ($i_H(\varphi'\circ\varphi)=i_H(\varphi')\cdot i_H(\varphi)$)
inducing an action of $G^k$ on $H^k$ by left multiplication ($\varphi\cdot h\doteq i_H(\varphi)\cdot h$) that satisfy (S.1-3) and such that we have for any $u\in T_eG^k$
$$T_{\varphi}i_H(u.\varphi)=T_{\id}i_H(u)\cdot i_H(\varphi)\,.$$ 
The same is true for $i_N$ and $N^k$. 
    \end{enumerate}
\end{prop}

\begin{proof}
    We already saw $\operatorname{Diff}_{C_0^k}$ satisfies the conditions (G.1-5). Furthermore, $C_0^{k-1}(\R^d,\R^{+}_*)$ and $C_0^{k-1}(\R^d,GL_d)$ are Lie groups. Furthermore, for $p,l\in \mathbb{N}$, the mapping $$ \begin{array}{ccl}
     C_0^{k-1+l}(\R^d,\R^p) \times G^k & \longrightarrow & C_0^{k-1}(\R^d,\R^p) \\
    (a,\varphi) &\longmapsto & a\circ\varphi \\
  \end{array}$$ is $C^l$ and $C^\infty$ with regards to the first variable. For $(\varphi,\omega)\in H^k$ (resp. $(\varphi,A)\in N^k$), its inverse is given by $(\varphi,\omega)^{-1} = (\varphi^{-1},\frac{1}{\omega\circ\varphi^{-1}})$ (resp. $(\varphi,A)^{-1} = (\varphi^{-1},A^{-1}\circ\varphi^{-1})$ and (G.2) is verified. Therefore $H^k$ and $N^k$ also satisfy (G.1-5).

  We directly see $i_H$ and $i_N$ are smooth morphisms of groups, and are embeddings. Let's compute their derivatives. For $(\varphi,\delta\varphi)\in TG^k$ we get
    \begin{align*}
        \partial_{\varphi}\left(|J\varphi|\right)\delta\varphi &=|\operatorname{det}(d\varphi)|\operatorname{tr}(d\varphi^{-1}d\delta\varphi) \\
        &= |J\varphi| \operatorname{tr}\left(d\left(\delta\varphi\circ\varphi^{-1}\right)\circ\varphi\right) \\
        &= |J\varphi| \operatorname{div}(\delta\varphi\circ\varphi^{-1})\circ\varphi
    \end{align*}
    Thus $T_{\varphi}i_H(\delta\varphi) = \left(\delta\varphi,\ \operatorname{div}\left(\delta\varphi\circ\varphi^{-1}\right)\circ\varphi\left|J\varphi\right|\right)$ and $T_{\varphi}i_N(\delta\varphi) = (\delta\varphi,J\delta\varphi)$. 

    The last point follows from the chain rule $|J(\varphi'\circ\varphi)| = |J\varphi'|\circ\varphi |J\varphi|$ which implies $i_H(\varphi'\circ\varphi)=i_H(\varphi')\cdot i_H(\varphi)$
\end{proof}

Since the family of groups $G^k$ act on $H^k$ and $N^k$ satisfying (S.1-3), we can apply proposition \ref{prop:19.6.1} and corollary \ref{cor:19.6.1}, and we have the following result :
\begin{prop}
    Let $k_0\geq 1$, and $V\hookrightarrow T_eG^{k_0+2}$ such that the hypotheses for the (GGA) framework are satisfied. Let $\varphi_0 \in G^{k_0}$ and $p^H_0=(p_0^{\varphi},p_0^{\omega})\in T_{i^H(\varphi_0)}^*H^{k_0}$ (resp. $p^N_0=(p^{\varphi}_0,p^{A}_0)\in T_{i^N(\varphi_0)}^*N^{k_0}$) such that $p_0^G=(Ti_H)^*p_0^H$ (resp. $p^G_0=(Ti_N)^*p_0^N$). Then we have:
\begin{enumerate}
    \item The normal geodesic in $H^{k_0}$ (resp. $N^{k_0}$) associated to the Hamiltonian equations with initial value $(i_{H}(\varphi_0),p^H_0)$ (resp. $(i_{N}(\varphi_0), p^{N}_0$)) stays in $i_H(G^{k_0})$ (resp. in $i_G(G^{k_0})$) and is a normal geodesic in the group $G^{k_0}$ starting from $(\varphi_0,p^G_0)$.     
    \item Furthermore, we have the following relation
    \begin{equation}
        p_t^G=(Ti_H)^*p_t^H\ \text{ (resp.  }
        p_t^G=(Ti_N)^*p_t^N)
    \end{equation}
\end{enumerate}
\end{prop}
\begin{proof}
    This is an immediate consequence of Corollary \ref{cor:19.6.1}. We just need to prove the relations between the co-state variables. Let $(\varphi_t,p_t^G)$ and $(h_t,p_t^H)$ the corresponding Hamiltonian geodesics such that $$m_G(\varphi_t,p_t^G) = m_H(h_t,p_t^H)$$ and $h_t=\varphi_t\cdot (\varphi_0,|J\varphi_0|) = i_H(\varphi_t)$, $p_t^H=(p_t^\varphi,p_t^\omega)$. Then for $\delta\varphi \in T_{\varphi_t}G^{k_0}$, we get 
    \begin{align*}
        \left(p_t^G\,|\,\delta\varphi\right) &= \left(m_G(\varphi_t,p_t^G)\,|\,\delta\varphi\circ\varphi_t^{-1}\right) 
        = \left(m_H(h_t,p_t^H)\,|\,\delta\varphi\circ\varphi_t^{-1}\right)\\
        &=\left(p_t^H\,|\,\delta\varphi\circ\varphi_t^{-1} \cdot i_H(\varphi_t) \right) 
        = \left(p_t^H\,|\,Ti_H(\delta\varphi)\right)
    \end{align*} 
    which gives the equality
    $
    p_t^G=(Ti_H)^*p_t^H
    $
\end{proof} 

\subsection{Regularity of the momentum}
In this section we arguments the introduction of the groups $H^{k_0}$ and $N^{k_0}$ by proving results on the regularity of the momentum. We first start with this general statement where we consider the Hamiltonian dual variable $p_t$ on $\mathcal{Q}$ where $\mathcal{Q}=G^{k_0}$ or $H^{k_0}$, or $N^{k_0}$ 
\begin{prop}
    If the momentum $p_t$ is $L^1$ at time $t=1$, then it stays $L^1$ for all previous times $t\leq 1$ and given by
    \begin{equation}
        p_t = p_1-\int_t^1 \partial_q(\xi_{q_s}(u_s))^*p_s ds
    \end{equation}
\end{prop}
\begin{proof}
The proof is done in the case where $\mathcal{Q} = G^{k_0} = \operatorname{Diff}_{C_0^{k_0}}(\R^d)$, the other cases are similar.
    We recall $p_t$ verifies the Hamiltonian equation with endpoint $p_1\in L^1(\R^d,\R^d)$, i.e. $p_t$ is solution of the Linear Cauchy problem : 
    $$
    \left\{
    \begin{array}{l}
        \dot{p}(t) = - \partial_q(\xi_{q_t}(u_t))^*p(t)  \\
        p(1) = p_1  
    \end{array}
    \right.
    $$
    which immediately implies we have :
    $$
    p_t =  p_1-\int_t^1 \partial_q(\xi_{q_s}(u_s))^*p_s ds
    $$
    Let $q\in G^k$, $u \in T_eG^{k_0+2} = C_0^{k_0+2}(\R^d,\R^d)$, the linear operator $p\in L^1 \mapsto \partial_q(\xi_{q}(u))^*p$ is bounded with image in $L^1$. Indeed, for $\delta q \in T_{q}\mathcal{Q}= C_0^{k_0}(\R^d,\R^d)$, we have :
    \begin{align*}
        \left(\partial_q(\xi_{q}(u))^*p\,|\,\delta q\right) &= \left(p\,|\,\partial_q(\xi_{q}(u)) \delta q \right) \\
        &= \int_{\R^d} \langle p(x), \partial_q(\xi_{q}(u)) \delta q \rangle_{\R^d} dx \\
        &= \int_{\R^d} \langle p(x), du(q(x)) \delta q(x) \rangle_{\R^d} dx \\
        &= \int_{\R^d} \langle du(q(x))^Tp(x), \delta q(x) \rangle_{\R^d} dx
    \end{align*}
    Since $du \in C_0^{k_0+1}(\R^d,\mathcal{M}_d(\R))$, therefore $\left(du\circ q\right)^T p$ is still in $L^1$, we have :
    \begin{align*}  
    |\partial_q(\xi_{q}(u))^*p|_{L^1}&\leq |du|_{C_0^{k+1}} |p|_{L^1} \\
    &\leq \operatorname{Cte} |u|_V|p|_{L^1}
    \end{align*}
    and the mapping $(u,q) \mapsto \partial_q(\xi_{q}(u))^*=(du\circ q)^T \in \operatorname{end}(L^1)$ is continuous, and therefore the momentum $p_t$ is $L^1$ for all $t\in [0,1]$.
\end{proof}
\begin{rem}
    We can adapt the previous proof and find that if $p_1$ is in $L^p$, for $p\geq 1$ (resp. in $C_0^l(\R^d,\R^d)$ for $l\leq k$), then $p_t$ stays in $L^p$ (resp. in $C_0^l(\R^d,\R^d)$) for all time.
\end{rem}
In the following, we will only consider the groups $G^{k_0}$ and $H^{k_0}$. We saw that when we solve the Hamiltonian equations in both $G^{k_0}$ and $H^{k_0}$, we get the following equality between the momenta 
\begin{equation}
    p_t^G = (Ti_H)^* p_t^H
\end{equation}
The embedding $i_H:G^{k_0}\rightarrow H^{k_0}$ is highly non surjective (the image in $H^{k_0}$ is not even dense), and therefore the dual mapping $(Ti_H)^* : T^*H^{k_0}\rightarrow T^*G^{k_0}$ is not injective. 

Let's compute the dual mapping. Let $p^G\in T^*G^{k_0}$, $p^H=(p^{\varphi},p^w) \in T^*H^{k_0}$, and $\delta\varphi\in TG^{k_0}$ we have $\left((Ti_H)^*p^H\,|\,\delta \varphi \right)=(p^\varphi\,|\,\delta\varphi) + (p^\omega\,|\,\operatorname{div}(\delta\varphi\circ\varphi^{-1})\circ\varphi |J\varphi|)$, i.e. we have the following (in the weak sense) : 
\begin{equation}
    (Ti_H)^*p^H = p^{\varphi} - \left|J\varphi\right| \left(d\varphi^*\right)^{-1}\nabla p^{\omega}  
\end{equation}
In particular, if we have $p^G = (Ti_H)^*p^H$, then we get 
\begin{equation}
    \label{rel_mom}
    p^G =  p^{\varphi} - \left|J\varphi\right| \left(d\varphi^*\right)^{-1}\nabla p^{\omega} 
\end{equation}
Obviously, the momentum $(p^{G},0)\in T^*H^{k_0}$ satisfies equation \eqref{rel_mom}. The non-injectivity of $(Ti_H)^*$ allows us to also consider $(p^\varphi,p^\omega)$ to gain regularity. For example, we know that if $\Omega$ is a bounded regular openset of $\mathbb{R}^d$ and $p\in L^2(\Omega)$ with $\operatorname{curl}(p)=0$, generalisation of the Poincaré's lemma \cite{kesavan2005poincare,girault2012finite} gives existence of $h\in H^1_0(\Omega)$ such that $\nabla h=p$. This suggest a corresponding gain of regularity in the case $\varphi=e_G$, between $p^G$ and $p^\omega$ solutions of \eqref{rel_mom}. 

In the inexact matching setting, we recall that the regularity of the momentum is given by the endpoint contraint. Suppose that the endpoint constraint is given by the $C^1$ mapping $U : (\varphi,\omega) \in H^{k_0} \mapsto U(\varphi,\omega)\in\R$, and let $\varphi_1\in G^{k_0}$ be the endpoint for the geodesic equation. Then the geodesics momenta for Hamiltonian in both $G^{k_0}$ and $H^{k_0}$ have endpoints given by 
$$
\left\{
\begin{array}{l}
     p_1^\varphi = \partial_\varphi U(\varphi_1,d\varphi_1)  \\
     p_1^\omega = \partial_\omega U(\varphi_1,d\varphi_1) \\
     p_1^G = (Ti_H)^*d_{i_H(\varphi_1)}U = d_{(\varphi_1,|J\varphi_1|)}U\circ T_{\varphi_1}i_H = p_1^{\varphi} - \left|J\varphi_1\right| \left(d\varphi_1^{*}\right)^{-1}\nabla p_1^{\omega}
\end{array}
\right.
$$

Therefore if $d U(\varphi_1,\omega_1)\in T^*H^{k_0}$ is $L^1$, the couple $(p_1^\varphi,p_1^\omega)$ is also $L^1$. However we will see in next subsection a case where $p_1^G$ (and therefore $p_t^G$) has singular parts coming from singularities from the target. This means the topological and geometric properties of targets and sources leads to the use of either groups $G^{k_0}$ or $H^{k_0}$.

\subsection{An application to transport on atlases}
We here give a motivation for using groups $H^{k_0}$ and $N^{k_0}$ to compute geodesics induced by the diffeomorphism actions. This will have particular applications in the special case of atlas (at a tissue scale for example), where the date is given as an image defined on a set of regions with discontinuities at the boundaries. In the following, we will consider a metric induced by a Hilbert space $V\hookrightarrow C_0^{k_0+2}(\R^d,\R^d)$.

Let $X_1, X_2,\ldots ,X_n \subset \R^d$ open connected pairwise disjoint subsets of $\R^d$ with piecewise $C^1$ boundaries such that $\R^d=\cup_{i=1}^n \overline{X}_i$. We denote $\Sigma = \cup_i \partial X_i$ the boundary. We consider a template image $I_0 : \R^d \rightarrow \R$ such that $I_{0|X_i}$ is smooth for each $X_i$, and the sets $I_0(X_i)$ are pairwise disjoint, and a smooth target image $I_1:\R^d\rightarrow \R$. We assume that both $I_0$ and $I_1$ are square integrable.

We derive the Hamiltonian equations in the space $G^{k_0}$ : 
\begin{equation}
    \left\{
    \begin{array}{ll}
        \dot{\varphi}_t = u_t \circ \varphi_t \\
        \dot{p}_t = - du_t^*\circ\varphi_t\,p_t \\
        \langle u_t, v \rangle_V = (p_t|v\circ\varphi_t ) 
    \end{array}
    \right.
\end{equation}

Now we also need to add an endpoint condition to perform inexact matching 
We define, for $I$ an image, the endpoint
$$
U(I) = \int_{\R^d} \left(I_1-I\right)^2 (y)dy = |I_1-I|^2_{L^2}
$$
or equivalently $U(\varphi) = U(\varphi \cdot I_0) = \int_{\R^d} \left(I_1-I_0\circ\varphi^{-1}\right)^2(y) dy = \int_{\R^d} \left(I_1\circ\varphi-I_0\right)^2(x) \left|J\varphi(x)\right| dx$ with $\varphi \in G^{k_0}$. 

\begin{prop}
    The endpoint constraint $U$ is $C^1$ and the momentum $p_1^{G^{k_0}}$ at time $1$ is given as a mixture
    \begin{equation}
        p_1 = -2\left(I_0-I_1\circ\varphi\right)\nabla (\varphi_1\cdot I_0) \circ\varphi_1|J\varphi_1|\lambda_d + \mathcal{J}\circ\varphi_1 \left|J\varphi_{1|\Sigma}\right| \mathcal{H}^{d-1}_{\Sigma}
    \end{equation}

    where $\mathcal{H}_{|\Sigma}^{d-1}$ the $d-1$ Hausdorff measure on the boundary, and, for $x\in\varphi(\Sigma)$, $|J\varphi_{1|\Sigma}|$ is the Jacobian of $\varphi_1$ restricted to the tangent bundle of $\Sigma$ and
    $$
    \mathcal{J}(x) = \lim_{\epsilon\rightarrow 0} \left(\left(I_1(x)-\varphi\cdot I_0(x+\epsilon\overset{\rightarrow}{n}(x))\right)^2 - \left(I_1(x)-\varphi \cdot I_0(x-\epsilon\overset{\rightarrow}{n}(x))\right)^2\right) \overset{\rightarrow}{n}(x)
    $$
    with $\overset{\rightarrow}{n}(x)$ unit normal vector on $\varphi(\Sigma)$.
\end{prop}

\begin{proof}
    We first start by proving $U$ is $C^1$ and we compute its derivative $dU$. Let $(\varphi, \delta\varphi)\in TG^{k_0} \simeq \operatorname{Diff}_{C_0^{k_0}}(\R^d)\times C_0^{k_0}(\R^d)$. 
    Since $I_1$ is smooth, we can differentiate under the integral sign and we have 
    \begin{align*}
         dU(\varphi)\delta\varphi &= \int_{\R^d}2 \left(I_1\circ\varphi-I_0\right) (x)\left\langle\nabla I_1\circ \varphi, \delta\varphi\right\rangle(x)  \left|J{\varphi}(x)\right| dx \\ &+ \int_{\R^d}\left(I_1\circ\varphi-I_0\right)^2 (x)\left|J\varphi\right| (x)\operatorname{div}(\delta\varphi\circ\varphi^{-1})\circ\varphi (x)dx \\
        &= \int_{\R^d}2 \left(I_1-\varphi\cdot I_0\right)(y)\left\langle\nabla I_1, \delta\varphi\circ\varphi^{-1} \right\rangle(y) dy + \int_{\R^d}\left|I_1-\varphi\cdot I_0\right|^2(y) \operatorname{div}(\delta\varphi\circ\varphi^{-1})(y) dy
    \end{align*}
Let's compute the second integral using Stokes theorem. We have :
\begin{align*}
    \int_{\R^d} &\left|I_1-\varphi\cdot I_0\right|^2(y) \operatorname{div}(\delta\varphi\circ\varphi^{-1})(y) dy = \sum_{i=1}^n \int_{\varphi(X_i)}\left|I_1-\varphi \cdot I_0\right|^2 (y)\operatorname{div}(\delta\varphi\circ\varphi^{-1})(y) dy \\
    & = \sum_{i=1}^n \left(-2 \int_{\varphi(X_i)} \left( I_1-\varphi\cdot I_0 \right) (y)\left\langle\nabla \left(I_1-\varphi\cdot I_0\right),\delta\varphi\circ\varphi^{-1}  \right\rangle(y) dy\right. \\ &  + \left. \int_{\partial \varphi(X_i)} \left|I_1(y)-\lim_{\epsilon\rightarrow 0}\varphi\cdot I_0(y-\epsilon \overset{\rightarrow}{n_i}(y))\right|^2 \left\langle \delta\varphi\circ\varphi^{-1}, \overset{\rightarrow}{n_i} \right\rangle (y)d\mathcal{H}^{d-1}_{\partial \varphi(X_i)}(y) \right) \\
    &= -2\int_{\R^d} \left( I_1-\varphi\cdot I_0 \right)\left\langle\nabla \left(I_1-\varphi\cdot I_0\right),\delta\varphi\circ\varphi^{-1}  \right\rangle(y) dy -\int_{\varphi(\Sigma)} \left\langle \mathcal{J}, \delta\varphi\circ\varphi^{-1} \right\rangle (y)d\mathcal{H}^{d-1}_{\varphi(\Sigma)}(y)
\end{align*}

Now as $p_1= - dU(\varphi_1)$, we get that :
\begin{equation}
    p_1 = -2\left(I_1\circ\varphi- I_0\right)\nabla (\varphi\cdot  I_0 )\circ\varphi|J\varphi|\lambda_d + \mathcal{J}\circ\varphi \left|J\varphi_{|\Sigma}\right| \mathcal{H}^{d-1}_{\Sigma}\,.
\end{equation}

\end{proof}

We see that the endpoint constraint can be written as $U(\varphi) = \tilde{U}(\varphi,|J\varphi|)$ where 
$$
\tilde{U}(\varphi,\omega) = \int_{\R^d}\left(I_1\circ\varphi(x) - I_0(x) \right)^2\omega(x) dx
$$
where $(\varphi,\omega)\in H^k$, and we can also compute the corresponding momentum $(p^{\varphi}_1,p^{\omega}_1)$ at time 1 which is $L^1$:
$$
\left\{\begin{array}{l}
     p^{\varphi}_1 = -\partial_{\varphi}\tilde{U}(\varphi_1,|J\varphi_1|) = 2\left(I_0-I_1\circ\varphi\right)\left|J\varphi_1 \right| \nabla I_1\circ\varphi \\
     p^{\omega}_1 = -\partial_{\omega}\tilde{U}(\varphi_1,|J\varphi_1|) = -|I_0-I_1\circ\varphi|^2
\end{array}\right.
$$


\section{Conclusion}
In this paper, we generalized the previous work on right-invariant sub-Riemannian structures on shape spaces induced by the group of diffeomorphisms \cite{arguillere2020sub,arguillere_trelat_2017,https://doi.org/10.48550/arxiv.1504.01767,ARGUILLERE2015139} to metrics induced by more general groups. In the (GGA) framework, we provided very general assumptions to study new group actions and capture more sophisticated deformations. \\

We also showed that the unreduction route, i.e. the over-parametrization through the use of extensions of the diffeomorphisms group, allows to gain some regularity on the momenta variables. Numerically, this disentanglement allows more degrees of freedom and we hope this can stabilize optimization and yield better solutions. The counterpart is that we need more variables and hence more memory for the new momenta variables. This gives new perspectives regarding optimization which are left for future investigations.

\appendix
\def\ACp{AC_{L^p}}
\def\AUi{\ACp(I_i,U_i)}
\def\AVi{\ACp(I_i,V_i)}
\def\B{\mathbb{B}}
\def\Bi{\B}
\def\Bip{\B}
\def\UBn{\mathcal{U}_{\mathcal{B}_n}}
\def\A{\mathcal{A}}
\def\An{\A_n}
\def\Ad{\A_2}
\def\UAn{\mathcal{U}_{\An}}
\def\teta{\tilde{\eta}}
\def\tetai{\teta_i}
\def\tetau{\teta_1}
\def\tetaip{\teta_{i+1}}
\def\trho{\tilde{\rho}}
\def\Ua{\mathcal{U}^\alpha}
\def\Ub{\mathcal{U}^\beta}
\def\Uc{\mathcal{U}^\gamma}
\def\Va{\mathcal{V}^\alpha}
\def\Vab{\mathcal{V}^{\alpha,\beta}}
\def\Vba{\mathcal{V}^{\beta,\alpha}}
\def\Ea{E^\alpha}
\def\Eb{E^\beta}
\def\Eab{E^{\alpha,\beta}}
\def\Eba{E^{\beta,\alpha}}
\def\Na{\mathcal{N}^\alpha}
\def\Nb{\mathcal{N}^\beta}
\def\Nab{\mathcal{N}^{\alpha,\beta}}
\def\jb{j^\beta}
\newtheorem{propA}{Proposition}[section]
\newtheorem{remA}{Remark}[section]
\section{Manifold structure on the set of absolutely continuous curves}
\label{app:A}
    In this part, we give a more detailed proof of proposition \ref{man_acg}, and put a Banach differentiable structure on the space of absolutely curves $AC_{L^p}([a,b],\mathcal{M})$, where $\mathcal{M}$ is differentiable manifold modeled on a Banach space $\mathbb{B}$. The proof is actually contained in \cite{Krikorian-1972}, in a more general setting, we include it here for the case of absolutely continuous curves for sake of completeness.

Let $\A$ be the set of all the $\alpha=(a^\alpha,U^\alpha,\varphi^\alpha)$ such that there exist a positive integer $n=n^\alpha$, a family $a^\alpha=(a^\alpha_i)_{0\leq i\leq n}$ such that $0=a_0<\ldots<a_n=1$ and two families $U^\alpha=(U^\alpha_i)_{1\leq i\leq n}$ and $\varphi^\alpha=(\varphi^\alpha_i)_{1\leq i\leq n}$ such that  $(U^\alpha_i,\varphi^\alpha_i)$ is a chart on $\mathcal{M}$ with $\varphi_i^\alpha:U^\alpha_i\to V^\alpha_i=\varphi^\alpha_i(U^\alpha_i)\subset \B$ for $1\leq i\leq n$. We will denote $I^\alpha=(I^\alpha_i)_{1\leq i\leq n}$ where $I^\alpha_i=[a^\alpha_{i-1},a^\alpha_i]$.
\\

For any $\alpha\in \A$ and $n=n^\alpha$, we denote $\Ua=\ACp(a^\alpha;U^\alpha)$ where
$$\ACp(a^\alpha;U^\alpha)=\{\ \eta\in C([0,1],\mathcal{M})\ |\ \eta(I^\alpha_i)\subset U^\alpha_i \text{ and }\varphi^\alpha_i\circ\eta_{|I^\alpha_i}\in \ACp(I^\alpha_i,\B),\ 1\leq i\leq n\ \}\,,$$
$\Ea=\prod_{i=1}^{n}\ACp(I^\alpha_i;\B)$ and $\Phi^\alpha:\Ua\to \Ea$ the one-to-one mapping defined by $$\Phi^\alpha(\eta)=(\varphi^\alpha_i\circ \eta_{|I^\alpha_i})_{1\leq i\leq n}\,.$$ 
We denote $\Na=\Phi^\alpha(\Ua)$.

\def\Uab{\mathcal{U}^{\alpha,\beta}}
\def\Uba{\mathcal{U}^{beta,\alpha}}
\def\Uap{\mathcal{U}^\alpha_\psi}
\def\Pap{\Phi^\alpha_\psi}

\begin{propA}
\label{man_ac}
    We can create a differentiable structure on $\ACp([0,1],\mathcal{M})$ by showing the two following facts:
\begin{enumerate}
\item \label{enum:21.9.1} For any $\alpha\in\A$, $\Na$ is a submanifold of $\Ea$.
\item \label{enum:21.9.2} For any $\alpha,\beta\in\A$ such that $\Uab=\Ua\cap\Ub\neq\emptyset$, then $\Phi^\alpha(\Uab)$ (resp.  $\Phi^\beta(\Uab)$) is an open subset of $\Na$ (resp. $\Nb$) and  $\Phi^{\alpha,\beta}=\Phi^\beta\circ(\Phi^\alpha)^{-1}:\Phi^\alpha(\Uab)\to \Phi^\beta(\Uab)$  is a smooth diffeomorphism.
\end{enumerate}
A smooth atlas is then given by the family of charts $(\Uap,\Pap)$, for any $\alpha\in\A$ and any chart $(\mathcal{V}_\psi,\psi)$ on $\Na$ where $\Uap=(\Phi^{\alpha})^{-1}(\mathcal{V}_\psi)$, and $\Pap=\psi\circ\Phi^\alpha$.
\end{propA}
\begin{remA}
    In this construction, the family of sets $(\mathcal{U}^{\alpha}_\psi)$ defines the manifold topology of $AC_{L^p}(I,\mathcal{M})$, and the sets $\mathcal{U}^{\alpha}$ are open submanifolds of $AC_{L^p}(I,\mathcal{M})$.
\end{remA}

\begin{proof}
    
We first show \eqref{enum:21.9.1}: For $n=1$, the result is trivial. For $n>1$, if we introduce the open set  $\Va=\{\ \eta\in\prod_{i=1}^n \ACp(I^\alpha_i;V^\alpha_i)\  |\ \eta_i(a_i)\in \varphi^\alpha_i(U^\alpha_{i+1})\ \forall 1\leq i\leq n-1\}$ of $\Ea$, we have $\Na\subset \Va$. Moreover, we have $\Na=(\sigma^\alpha)^{-1}(0)$ for $\sigma^\alpha:\Va\to \prod_{i=2}^{n}\B$ such that $\sigma^\alpha(\eta)=(\varphi^{\alpha}_{i+1}\circ(\varphi^\alpha_i)^{-1}\circ\eta_i(a^\alpha_i)-\eta_{i+1}(a^\alpha_i))_{1\leq i <n}$. Since $\sigma^\alpha$ is smooth,  it is enough to conclude to show that at any point $\eta\in \Na$, $\sigma^\alpha$ is a submersion at $\eta$. However, considering the closed subspace $H_n=\{\ (0,\delta \eta_2,\cdots,\delta \eta_{n})\ |\ \delta \eta_i:I^\alpha_i\to\B \text{ is contant}\ \}$ of $ T_{\eta} \Va\simeq \prod_{i=1}^{n} \ACp(I^\alpha_i,\B)$, we get immediatly that $T_{\eta} \sigma^\alpha $ is surjective and $\text{Ker}\,T_{\eta}\sigma^\alpha\oplus H_n= T_\eta\Va$ so that its kernel splits (see \cite{lang2001fundamentals} prop 2.2)\\
\def\ja{j^\alpha}

\begin{center}
  \begin{figure}[h]
    \includegraphics[scale=0.25]{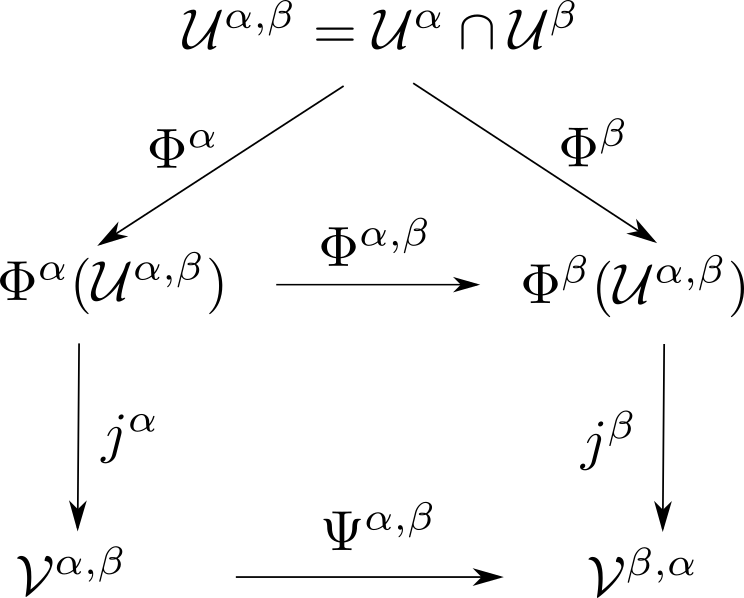}
  \end{figure}
\end{center}
Let us prove \eqref{enum:21.9.2}: Let first remark that there exists a unique $\gamma=(a^\gamma,U^\gamma)$ such that $\text{range}(a^\gamma)=\text{range}(a^\alpha)\cup \text{range}(a^\beta)$, $\Uab=\Uc$ with $U^{\gamma}_k=U^\alpha_{i_k}\cap U^\beta_{j_k}$ for $(i_k,j_k)$ the unique pair $(i,j)$ for which $I^\gamma_k=I^\alpha_i\cap I^\beta_j$. We check that $\Phi^\alpha(\Uab)=\Na\cap \prod_{i=1}^{n^\alpha}\mathcal{V}^{\alpha}_i$ with $\mathcal{V}^{\alpha}_i=\{\ \eta_i\in\ACp(I^\alpha_i,\B)\ |\ \eta_i(I^\gamma_k)\subset \phi^\alpha_i(U^\gamma_k),\ k\text{ s.t. }i=i_k\ \}$ open in $\ACp(I^\alpha_i,\B)$ so that $\Phi(\Uab)$ is an open set of $\Na$.

We consider the continuous linear injectif mapping  $\ja:\Ea\to \Eab=\prod_{1\leq k\leq n^\gamma}\ACp(I^\gamma_k,\B)$ defined by $\ja(\eta)=({\eta_{i_k}}_{|I^\gamma_k})_{1\leq k\leq n^\gamma}$. We check that $\ja$ is an isomorphism onto its image since $\Eab=\ja(\Ea)\oplus \prod_{k=1}^{n^\gamma}\{\ \delta\eta_k\in \ACp(I^\gamma_k,\B)\ |\ \delta\eta_k\equiv c\bm{1}_{i_k=0}\text{ for }c\in\R\ \}$ so that $\ja(\Ea)$ splits (see \cite{lang2001fundamentals} prop 2.2). Similarly $\jb:\Eb\to \Eba$ is an isomorphism onto its image. Moreover, we have a smooth diffeomorphic mapping $\Psi_{\alpha,\beta}:\Vab\to\Vba$ between the open set $\Vab=\prod_{k=1}^{n^\gamma}\ACp(I^\gamma_k,\varphi_{i_k}^\alpha(U^\gamma_k))\subset E^{\alpha,\beta}$ and the open set $\Vba=\prod_{k=1}^{n^\gamma}\ACp(I^\gamma_k,\varphi_{j_k}^\beta(U^\gamma_k))\subset\Eba$ given by $\Psi^{\alpha,\beta}(\eta=(\eta_k)_{1\leq k\leq n^\gamma})=(\varphi^\beta_{j_k}\circ(\varphi^{\alpha}_{i_k})^{-1}\circ\eta_k)_{1\leq k\leq n^\gamma}$. Since $\ja(\Phi^\alpha(\Uab))\subset \Vab$ and  for any $\eta\in \Uab$ we have $\Psi^{\alpha,\beta}\circ\ja\circ\Phi^\alpha(\eta)= \jb\circ\Phi^{\beta}(\eta)$ we get that $\Phi^{\alpha,\beta}=(\jb)^{-1}\circ\Psi^{\alpha,\beta}\circ\ja_{|\Phi^\alpha(\Uab)}$ so that $\Phi^{\alpha,\beta}$ is a smooth diffeomorphism. 
\end{proof}

We now want to characterize the tangent bundle of $AC_{L^p}(I,\mathcal{M})$. We first start by proving the following result
\begin{propA}
\label{propA:ev}
   The evalutation map $\operatorname{ev}_{t_0} : AC_{L^p}(I,\mathcal{M})\rightarrow \mathcal{M}$, where $t_0\in I$ is smooth.
\end{propA}
\begin{proof}
    We prove this result using the previous construction. Let $t_0\in I$, let $\eta \in AC_{L^p}(I,\mathcal{M)}$, and let $\mathcal{U}^{\alpha} = AC_{L^p}(a^{\alpha};U^{\alpha})$ such that $\eta$ is in $\mathcal{U^{\alpha}}$. Let also $i\in\{1,..,n\}$ such that $t_0\in I_i^\alpha$, and thus
    $$
\operatorname{ev}_{t_0}(\mathcal{U}^\alpha)\subset U_i^\alpha
    $$
    It is equivalent to prove the smoothness of the induced evolution mapping $\tilde{\operatorname{ev}}_{t_0} = \varphi_i^\alpha \circ \operatorname{ev}_{t_0}\circ (\Phi^\alpha)^{-1} : \mathcal{N}^\alpha\rightarrow V_i^\alpha$, and given by :
    $$\tilde{\operatorname{ev}}_{t_0} : \left\{
    \begin{array}{ccl}
        \mathcal{N}^\alpha & \longrightarrow & V_i^\alpha\\
         (\eta_1,\ldots,\eta_n) & \longmapsto & \eta_i(t_0)
    \end{array}\right.
    $$
    This mapping is the restriction of a linear mapping $\tilde{\operatorname{ev}}_{t_0}:E^\alpha\rightarrow \mathbb{B}$, which is continuous since the evaluation is continuous in $AC_{L^p}(I_i^\alpha,\mathbb{B})$ \cite{Glo}, and thus $\tilde{\operatorname{ev}}:\mathcal{N}^\alpha\rightarrow V_i^\alpha$ is smooth. 
\end{proof}

Now we can finally characterize the tangent bundle of $AC_{L^p}(I,\mathcal{M})$, and we prove it is isomorph to the bundle $AC_{L^p}(I,T\mathcal{M})\rightarrow AC_{L^p}(I,T\mathcal{M})$

\begin{propA}
\label{propA:tan}
    Let $\pi_* : AC_{L^p}(I,T\mathcal{M}) \rightarrow AC_{L^p}(I,\mathcal{M})$ given by
    $$
\pi_* :\left\{
\begin{array}{ccl}
     AC_{L^p}(I,T\mathcal{M}) & \longrightarrow & AC_{L^p}(I,\mathcal{M})  \\
     \gamma & \longmapsto & \pi\circ \gamma
\end{array}
\right.
    $$ where $\pi=\pi_{T\mathcal{M}} : T\mathcal{M} \rightarrow \mathcal{M}$ is the tangent bundle of $\mathcal{M}$. Then $\pi_* : AC_{L^p}(I,T\mathcal{M}) \rightarrow AC_{L^p}(I,\mathcal{M})$ can be taken as the tangent bundle of $AC_{L^p}(I,\mathcal{M})$ through the isomorphism
    \begin{equation}
        \zeta : \left\{ 
        \begin{array}{ccl}
             TAC_{L^p}(I,\mathcal{M}) & \longrightarrow & AC_{L^p}(I,T\mathcal{M})  \\
             w & \longmapsto & [t\mapsto T_{\pi(w)}\operatorname{ev}_t w ]
        \end{array}
        \right.
    \end{equation}
\end{propA}
\begin{proof}
    Since the evaluation mapping is smooth, the curve $\gamma : t\mapsto T_{\pi(w)}\operatorname{ev}_t w$ is well defined in $T\mathcal{M}$. Let's first show that it is an absolutely continuous curve. We denote $\eta = \pi(w)$, and we have $\eta(t) = \operatorname{ev}_t(\eta) = \pi(\gamma(t))$. Like previously in proposition \ref{man_ac}, we can define an open submanifold $(\mathcal{U}^{\alpha,\mathcal{M}},\Phi^{\alpha,\mathcal{M}})$ of $AC_{L^p}(I,\mathcal{M})$, where $\mathcal{U}^{\alpha,\mathcal{M}}=AC_{L^p}(a^\alpha,U^{\alpha,\mathcal{M}})$ for some open charts $(U_i^{\alpha,\mathcal{M}},\varphi_i^{\alpha,\mathcal{M}})$ of $\mathcal{M}$. We suppose $\eta \in \mathcal{U}^{\alpha,\mathcal{M}}$. We can now also define an induced submanifold $(\mathcal{U}^{\alpha,T\mathcal{M}},\Phi^{\alpha,T\mathcal{M}})$ of $AC_{L^p}(I,T\mathcal{M})$, where $\mathcal{U}^{\alpha,T\mathcal{M}}=AC_{L^p}(a^\alpha,U^{\alpha,T\mathcal{M}})$, $U_i^{\alpha,T\mathcal{M}} = TU_i^{\alpha,\mathcal{M}}$, and $\varphi_i^{\alpha,T\mathcal{M}} = T\varphi_i^{\alpha,\mathcal{M}}$. The diffeomorphism $\Phi^{\alpha,T\mathcal{M}} : \mathcal{U}^{\alpha,T\mathcal{M}}\rightarrow \mathcal{N}^{\alpha,T\mathcal{M}}$ is defined as previously with the diffeomorphisms $\varphi_i^{\alpha,T\mathcal{M}}$. This gives us the following commutative diagram 
    $$
\xymatrix{
     \mathcal{U}^{\alpha,\mathcal{M}} \ar[d]^-{\Phi^{\alpha,\mathcal{M}}} & \ar[l]_{\pi_*} \mathcal{U}^{\alpha,T\mathcal{M}}  \ar[d]^-{\Phi^{\alpha,T\mathcal{M}}} \\
    \mathcal{N}^{\alpha,\mathcal{M}} & \mathcal{N}^{\alpha,T\mathcal{M}}
  }
    $$
Since $\eta\in \mathcal{U}^{\alpha,\mathcal{M}}$, and $\eta(t)=\pi(\gamma(t))$, therefore for all $1\leq i \leq n$, $\gamma(I_i^\alpha)\subset U_i^{\alpha,T\mathcal{M}} = TU_i^{\alpha,\mathcal{M}}$, such that $T\varphi_i^\alpha\circ \gamma_{|I_i^\alpha}$ is in $TV_i^\alpha$ and for $t\in I_i^\alpha$ 
\begin{align*}
    T\varphi_i^\alpha\circ \gamma_{|I_i^\alpha}(t) &= T (\varphi_i^\alpha\circ\operatorname{ev}_t) w \\
    &= T (\varphi_i^\alpha\circ\operatorname{ev}_t\circ(\Phi^{\alpha,\mathcal{M}})^{-1}) T\Phi^{\alpha,\mathcal{M}}w
\end{align*}
Since $\tilde{w}=T\Phi^{\alpha,\mathcal{M}}w\in T\mathcal{N}^\alpha \subset \prod_{i=1}^nAC_{L^p}(I_i^\alpha,V_i^\alpha)\times AC_{L^p}(I_i^\alpha,\mathbb{B})$, we thus have $T\varphi_i^\alpha\circ \gamma_{|I_i^\alpha}(t)=\tilde{w}(t)$, and then $\gamma_{|I_i^\alpha}\in AC_{L^p}(I_i^\alpha,U_i^{\alpha,T\mathcal{M}})$ and $\gamma$ is continuous. Therefore $\gamma$ is in $\mathcal{U}^{\alpha,T\mathcal{M}}\subset AC_{L^p}(I,T\mathcal{M})$, and $\zeta(T\mathcal{U}^{\alpha,\mathcal{M}})\subset \mathcal{U}^{\alpha,T\mathcal{M}}$. \\

Let's prove now $\zeta : T\mathcal{U}^{\alpha,\mathcal{M}}\rightarrow \mathcal{U}^{\alpha,T\mathcal{M}}$ is smooth. According to what precedes, we get the following diagram
$$
\xymatrix{
     T\mathcal{U}^{\alpha,\mathcal{M}} \ar[dd]^-{T\Phi^{\alpha,\mathcal{M}}}  \ar[r]^{\zeta} & \mathcal{U}^{\alpha,T\mathcal{M}}  \ar[dd]^-{\Phi^{\alpha,T\mathcal{M}}} \\ \\
    T\mathcal{N}^{\alpha,\mathcal{M}} \ar@{^{(}->}[d] & \mathcal{N}^{\alpha,T\mathcal{M}} \ar@{^{(}->}[d] \\
    TE^\alpha  \ar[r]^-{\sim} & \prod_{i=1}^n AC_{L^p}(I_i^\alpha,\mathbb{B}\times\mathbb{B})
  }
    $$
And we can verify that under the identification $TE^{\alpha}\simeq\prod_{i=1}^n AC_{L^p}(I_i^\alpha,\mathbb{B}\times\mathbb{B})$, we also have the identification $T\mathcal{N}^{\alpha,\mathcal{M}}\simeq\mathcal{N}^{\alpha,T\mathcal{M}}$, and this identification is exactly given by $\Phi^{\alpha,T\mathcal{M}}\circ\zeta\circ (T\Phi^{\alpha,\mathcal{M}})^{-1}$. Therefore $\zeta : T\mathcal{U}^{\alpha,\mathcal{M}} \rightarrow \mathcal{U}^{\alpha,T\mathcal{M}}$ is a smooth bundle isomorphism with base space $\mathcal{U}^{\alpha,\mathcal{M}}$
\end{proof}

\bibliographystyle{abbrv}
\bibliography{ref}

\end{document}